\documentclass[11pt,reqno]{amsart}
\usepackage[utf8]{inputenc}
\usepackage[T1]{fontenc}
\usepackage{amssymb}
\usepackage{pdfsync}
\usepackage{enumitem}
\usepackage{tikz-cd}
\usepackage[english]{babel}
\usepackage[left= 1.3 in, right= 1.3 in ,top= 1 in, bottom = 2 in]{geometry}

\numberwithin{equation}{section}

\usepackage[pdftex,pagebackref,colorlinks=true,urlcolor=blue,linkcolor=blue,citecolor=blue]{hyperref}
\usepackage{graphicx}
\usepackage[normalem]{ulem}
\usepackage{xcolor}
\usepackage[capitalize]{cleveref}
\usepackage{fullpage}
\usepackage{color}
\usepackage{amsmath}
\usepackage{amsfonts}
\usepackage{mathrsfs}
\usepackage{t1enc , graphicx}
\usepackage{verbatim}
\usepackage{bbm}
\usepackage[colorinlistoftodos]{todonotes}
\usepackage{enumitem}
\usepackage{mathtools}

\newcommand{\Arg}{\operatorname{Arg}}

\newcommand{\R}{\mathbb{R}}

\renewcommand{\epsilon}{\varepsilon}
\newcommand{\C}{\mathbb{C}}
\renewcommand{\d}{\,\mathrm{d}}

\newcommand{\NN}{\mathcal{N}}
\newcommand{\G}{\mathbb{G}}
\newcommand{\Z}{\mathbb{Z}}

\newcommand{\N}{\mathbb{N}}
\newcommand{\D}{\mathbb{D}}
\newcommand{\B}[1]{S_{#1}}
\newcommand{\oldphi}[2]{\NN\big(\gcd(#1,#2)\big)-1}

\renewcommand{\P}{\mathbb{P}}

\renewcommand{\Re}{\mathrm{Re}\,}
\renewcommand{\Im}{\mathrm{Im}\,}
\renewcommand{\i}{\mathrm{i}}
\newcommand{\LL}[1]{\log #1}
\newcommand{\E}{\mathop{\mathbb{E}}}
\newtheorem{theorem}{Theorem}[section]
\newtheorem{proposition}[theorem]{Proposition}

\newtheorem{lemma}[theorem]{Lemma}
\newtheorem{corollary}[theorem]{Corollary}

\theoremstyle{definition}
\newtheorem{definition}[theorem]{Definition}
\newtheorem*{definition*}{Definition}
\newtheorem{question}[theorem]{Question}
\newtheorem*{question*}{Question}
\newcounter{proofcount}
\AtBeginEnvironment{proof}{\stepcounter{proofcount}}% count the proofs

\makeatletter                  % reset the claim counter each time proofcount is
\@addtoreset{claim}{proofcount}% increased, effectively this restarts the claims
\makeatother

\theoremstyle{remark}

\newtheorem{example}[theorem]{Example}
\newtheorem{remark}[theorem]{Remark}

\theoremstyle{plain}
\newcounter{MainTheoremCounter}

\newtheorem{Maintheorem}[MainTheoremCounter]{Theorem}

\author{Sebasti\'an Donoso}
\author{Anh N. Le}
\author{Joel Moreira}
\author{Wenbo Sun}
\address{Departamento de Ingenier\'{\i}a Matem\'atica and Centro de Modelamiento Matem\'atico, Universidad de Chile \& IRL-CNRS 2807, Beauchef 851, Santiago, Chile.}
\email{sdonoso@dim.uchile.cl}
\address{Department of Mathematics\\
	University of Denver\\
	2390 S. York St., Denver, CO, 80210, USA}
\email{anh.n.le@du.edu}

\address{Mathematics Institute\\ University of Warwick\\
Coventry, UK}
\email{joel.moreira@warwick.ac.uk}

\address{Department of Mathematics\\ Virginia Polytechnic Institute and State University\\ 225 Stanger Street,
	Blacksburg, VA, 24061-1026, USA}
\email{swenbo@vt.edu}

\thanks{S.D. was partially funded by Centro de Modelamiento Matemático (CMM) FB210005 BASAL funds for centers of excellence from ANID-Chile and ANID/Fondecyt/1200897. W.S. was partially supported by the NSF Grant DMS-2247331}


\title{Averages of completely multiplicative functions over the Gaussian integers -- a dynamical approach}

\begin{document}
\begin{abstract}
We prove a pointwise convergence result for additive ergodic averages associated with certain multiplicative actions of the Gaussian integers.  
We derive several applications in dynamics and number theory, including:
\begin{enumerate}[label=(\roman*)]
       \item Wirsing's theorem for Gaussian integers: if $f\colon \mathbb{G} \to \mathbb{R}$ is a bounded completely multiplicative function, then the following limit exists:
       \[
        \lim_{N \to \infty} \frac{1}{N^2} \sum_{1 \leq m, n \leq N} f(m + \i n).
       \]
       
       \item An answer to a special case of a question of Frantzikinakis and Host: for any completely multiplicative real-valued function $f: \mathbb{N} \to \mathbb{R}$, the following limit exists:
       \[
            \lim_{N \to \infty} \frac{1}{N^2} \sum_{1 \leq m, n \leq N} f(m^2 + n^2).
       \]
       
       \item A variant of a theorem of Bergelson and Richter on ergodic averages along the $\Omega$ function: if $(X,T)$ is a uniquely ergodic system with unique invariant measure $\mu$, then for any $x\in X$ and $f\in C(X)$,
\[
    \lim_{N\to\infty}\frac{1}{N^2}\sum_{1 \leq m, n \leq N} f(T^{\Omega(m^2 + n^2)}x)=\int_Xf \ d\mu.
\]
\end{enumerate}

\end{abstract}

\maketitle

\tableofcontents
 
\section{Introduction}

\subsection{Wirsing theorems for Gaussian integers}

A function $f: \N \to \C$ is called \emph{multiplicative} if $f(mn) = f(m) f(n)$ whenever $m, n\in\N$ are co-prime; and $f$ is \emph{completely multiplicative} if this relation holds for every $m, n \in \N$. 
The statistical behavior of multiplicative functions is a central topic in analytic number theory and many important theorems in this area can be recast into the language of multiplicative functions. 
For instance, the Liouville function $\lambda: \N \to \{-1, 1\}$ defined as $\lambda(n) = (-1)^{\Omega(n)}$, where $\Omega(n)$ is the number of prime factors of $n$ counting with multiplicities, is a completely multiplicative function. It is well known that the Prime Number theorem is equivalent to 
\begin{equation}\label{eq_PNTinZ}
     \lim_{N \to \infty} \frac1N\sum_{n=1}^N \lambda(n) = 0
\end{equation}
(see, for example,  \cite[Page 96]{Chowla-Riemann-hypothesis}).
The harder part of establishing \eqref{eq_PNTinZ} is to show that the limit exists.
Generalizing this fact, Erd\H{o}s and Wintner \cite{Erdos-Some-unsolved-problems-1957-Michigan-Math} conjectured that for any multiplicative function $f:\N\to\{-1, 1\}$, the limit
\begin{equation}\label{eq_wirsinginZ}
     \lim_{N \to \infty} \frac1N\sum_{n=1}^N f(n)
\end{equation}
exists.
Around 1961, theorems of Delange \cite{Delange-surles} and Wirsing \cite{Wirsing_61} gave a satisfactory answer for multiplicative functions with non-zero mean values. 
The general case of the Erd\H{o}s-Wintner conjecture, which contains a proof of the Prime Number Theorem, was only established by Wirsing \cite{Wirsing-1967} in 1967.
In fact, Wirsing's theorem states that the limit \eqref{eq_wirsinginZ} exists for any bounded real valued multiplicative function.
A celebrated result of Hal\'asz \cite{Halasz-1968} in 1968 further extended the analysis to complex valued functions, where the picture is complicated by the fact that the limit \eqref{eq_wirsinginZ} does not always exist.

There are several possible ways to strengthen \eqref{eq_PNTinZ} or Wirsing's theorem.
For example, a conjecture by Chowla \cite[Page 96]{Chowla-Riemann-hypothesis} states that if $P \in \Z[x]$ is a polynomial satisfying $P \neq c Q^2$ for every $c \in \Z, Q \in \Z[x]$, then
\begin{equation}\label{eq:chowla_single}
    \lim_{N \to \infty} \frac1N\sum_{n=1}^N \lambda(P(n)) = 0.
\end{equation}
This conjecture is still wide open despite a large number of developments seen in the last decade. 
A survey of much of the progress can be found in \cite{Ferenczi_Kulaga-Przymus_Lemanczyk18} and references therein.

Averages of multiplicative functions over one variable, such as in \eqref{eq:chowla_single}, are notoriously hard to analyze. However, their multivariate counterpart seems to be more tractable.
In this direction, another conjecture, also attributed to Chowla states that if $P\in\Z[x,y]$ and $P \neq c Q^2$ for every $c \in \Z, Q \in \Z[x, y]$ then
\begin{equation}\label{eq:chowla_multi}
    \lim_{N \to \infty}\frac1{N^2}\sum_{m, n=1}^N \lambda(P(m,n)) = 0.
\end{equation}
(See \cite[Equation (1.2)]{Helfgott-Parity-reducible}.)
When $\deg(P)=2$, this conjecture was verified by Helfgott \cite{Helfgott-thesis}, based on ideas of de la Vall\'ee-Poussin \cite{VallePoussin:1896,VallePoussin:1897}. Helfgott later extended his analysis to cover the case $\deg(P) = 3$ in \cite{Helfgott-Parity-reducible}.
More recently, Green and Tao \cite{Green_Tao10} established \eqref{eq:chowla_multi} when $P$ is a product of pairwise independent linear forms.

In place of the Liouville function $\lambda$, a similar question can be asked about an arbitrary completely multiplicative function. 
In this direction, Frantzikinakis and Host \cite{Frantzikinakis_Host_2017} established the analogous statement to \eqref{eq:chowla_multi} for any ``aperiodic'' multiplicative function and for a class of polynomials $P$ which includes all products of pairwise independent linear forms.
They later posed the following question, which was a major motivator for our current paper.

\begin{question}[{\cite[Page 91]{Frantzikinakis-Host_Asymp}}]\label{question_FH}
    Let $f: \N \to \R$ be a real valued bounded completely multiplicative function and let $P \in \Z[x,y]$ be a homogeneous polynomial with values on the positive integers. Does the limit
    \[
       \lim_{N \to \infty}\frac1{N^2}\sum_{m, n=1}^N f(P(m,n))
    \]
    exist?
    \label{ques:frant-host}
\end{question}

Already in \cite{Frantzikinakis-Host_Asymp},  Frantzikinakis and Host provided a positive answer to \cref{ques:frant-host} in the special case when the polynomial $P$ is a product of linear forms. 
Shortly after, Klurman and Mangerel \cite{Klurman-Mangerel-effective} obtained a concrete, effective formula for the averages in that case. 
A related work was also done by Matthiesen in \cite{Matthiesen-linearcorrelations}.
Nevertheless, when $P$ is not a product of linear forms, the answer to \cref{ques:frant-host} remains elusive,  and even solving it for specific functions $f$ poses significant challenges. 

Our first main result answers \cref{ques:frant-host} for the polynomial $P(m,n)=m^2+n^2$ and for an arbitrary completely multiplicative function $f$.
Given a function $f: \{1, \ldots, N\} \to \C$, we write $\E_{1 \leq m, n \leq N} f(x)$ for the average $\frac{1}{N^2} \sum_{1 \leq m, n \leq N} f(x)$. More generally, if $A$ is a finite set and $f: A \to \C$ is a function on $A$, we use $\E_{x\in A} f(x)$ as a shorthand notation for $ \frac{1}{|A|}\sum_{x\in A} f(x)$.

\begin{Maintheorem}
\label{thm:convergence_integer}
    Let $f: \N \to \R$ be a bounded completely multiplicative function. Then the average
    \begin{equation*}\label{eq_thm:convergence_integer}
        \lim_{N \to \infty} \E_{1\leq m, n\leq N} f(m^2 + n^2)
    \end{equation*}
    exists and equals 
    \begin{equation}\label{eq:Pf_expanded_A}
    \frac1{2-f(2)}\cdot\prod_{\substack{p \text{ prime } \\ p\equiv1\bmod 4}} \left(\frac{p - 1}{p - f(p)}\right)^2\cdot\prod_{\substack{p \text{ prime } \\ p \equiv 3 \bmod 4}} \frac{p^2 - 1}{p^2 - f(p)^2}.
    \end{equation}

\end{Maintheorem}

A main tool used in the proof of Frantzikinakis and Host's result \cite{Frantzikinakis-Host_Asymp} mentioned above is a structure theorem of multiplicative functions developed in \cite{Frantzikinakis_Host_2017} which roughly speaking says that any bounded multiplicative function can be decomposed into a component that resembles a periodic function and a Gowers--uniform error term (see also \cite{Sun1,Sun2}). 
Although this structure theorem provides an effective way to deal with multiplicative functions along linear forms, it does not seem to help for general higher degree polynomials. 
An alternative approach to handle certain higher degree polynomials is to realize them as norm forms of number fields.
In this paper we focus on the polynomial $P(m,n)=m^2+n^2$, which can naturally be viewed as the norm function over the Gaussian integers $\G: = \{m + \i n: m, n \in \N\}$.
Recall that the norm function over $\G$ is $\NN(m + \i n) = m^2 + n^2$ and it satisfies $\NN(ab) = \NN(a) \NN(b)$ for any $a, b \in \G$.
Therefore, given a completely multiplicative function $f:\N\to\R$, the composition $f \circ \NN$ is a completely multiplicative function from the set of non-zero Gaussian integers $\G^*$ to $\R$.
Using this observation, we are able to derive \cref{thm:convergence_integer} from a version of Wirsing's theorem for Gaussian integers.

\begin{theorem}
\label{thm:convergence_Gaussian1}
    If $f: \G^* \to \R$ is a real-valued bounded completely multiplicative function, then the limit
    \[\lim_{N \to \infty} \E_{1\leq m, n\leq N} f(m + \i n)\] exists.
\end{theorem}

It is possible to identify the limit in \cref{thm:convergence_Gaussian1} as an Euler product.
To write this product, denote by $\P$ the set of Gaussian primes and by $\P_1$ the set of Gaussian primes in the first quadrant. (We discuss basic properties of Gaussian primes in \cref{sec:background}).
Given a bounded completely multiplicative function $f:\G^*\to\C$, we define $P(f)$ as
\begin{equation}\label{define:P_f}
    P(f) \coloneqq  \prod_{p \in \P_1}\frac{\NN(p) - 1}{\NN(p) - f(p)}.
\end{equation}
We remark that the infinite product defining $P(f)$ does not necessarily converge for every bounded completely multiplicative function $f:\G^*\to\C$; however, it does when $f$ takes real values (see \cref{lem:existence_Pf}).  

At this stage we move from averages over squares to the more general situation of averages over what we call \emph{dilated F{\o}lner sequences}. 
This concept will be introduced and discussed in detail in \cref{sec:def_dilated_Folner_seq}; for now, it suffices to say that examples of dilated F{\o}lner sequences include the sequence of squares $\Phi_N = \{m + \i n\in\G^*: 1 \leq m, n \leq N\}$ and the sequence of discs $\Phi_N = \{m + \i n\in\G^*: 0<m^2 + n^2 < N^2\}$.
We can now formulate a stronger version of \cref{thm:convergence_Gaussian1}.

\begin{Maintheorem}
\label{thm:convergence_Gaussian}
    If $f: \G^* \to \R$ is a real-valued bounded completely multiplicative function satisfying $f(\i)=1$ and $(\Phi_N)_{N \in \N}$ is a dilated F{\o}lner sequence, then $\lim_{N \to \infty} \E_{n \in \Phi_N} f(n)$ exists and equals $P(f)$.
\end{Maintheorem}

\begin{remark}\label{remark:restriction_fi}

The assumption $f(\i)=1$ is important in \cref{thm:convergence_Gaussian}: for instance if $\Phi$ is a finite set invariant under multiplication by $\i$ (such as a centered square or a  disc), the average $\E_{n\in\Phi}f(n)$ equals $0$ whenever $f(\i)\neq1$.
However, the condition $f(\i)=1$ is not necessary for the limit to exist.
In \cref{thm:no_restrictionC} below we derive from \cref{thm:convergence_Gaussian} a more general version without the assumption that $f(\i)=1$. 

\end{remark}

\begin{remark}
\cref{thm:convergence_Gaussian} is false for general (i.e., ``non-dilated'') F{\o}lner sequences. 
More precisely, in \cref{sec:def_dilated_Folner_seq}, we show that there exists a completely multiplicative function $f: \G^* \to \{-1, 1\}$ and an additive F{\o}lner sequence $(\Phi_N)_{N \in \N}$ such that $\lim_{N \to \infty} \E_{n \in \Phi_N} f(n)$ does not exist. 
Furthermore, a modification of an argument by Fish \cite{Fish-normal_Liouville} shows that if one considers a random completely multiplicative function $f: \G^* \to \{-1, 1\}$, then almost surely, $f$ will contain all finite patterns of $-1$ and $1$. 
Details are given in \cref{sec:counter_example_non_dilated}.
\end{remark}

Hal\'asz's theorem mentioned above, which contains Wirsing's theorem as a special case, was extended to arbitrary function fields by Granville, Harper, and Soundararajan \cite{Granville-Harper-Soundararajan-2015}. 
On the other hand, as far as we know, there is no analogue of Wirsing's and Hal\'asz's theorems in the number field setting. 
\cref{thm:convergence_Gaussian} partially fills this gap because it can be seen as an analogue of Wirsing's theorem for completely multiplicative functions on the Gaussian integers.
In fact, we prove a more general result, which partially explains why the restriction that $f$ takes on real values is convenient in Wirsing's theorem.

For $z \in \C \setminus \{0\}$, let $\Arg(z) \in [-\pi, \pi)$ denote its argument, and define $\Arg(0) = 0$.
Any real number $x$ has $\Arg(x) = 0$ or $-\pi$, so \cref{thm:convergence_Gaussian} is a special case of the following theorem.

\begin{Maintheorem}\label{thm:generalized_Wirsing}
    Let $f: \G^*\to \C$ be a bounded completely multiplicative function such that $f(\i)=1$ and $\Arg(f(\P))$ is finite.
    Then for every dilated F{\o}lner sequence $(\Phi_N)_{N \in \N}$, the limit $\lim_{N \to \infty} \E_{n \in \Phi_N} f(n)$ exists and equals $P(f)$.
\end{Maintheorem}

    Similarly to \cref{remark:restriction_fi}, it is possible to remove from \cref{thm:generalized_Wirsing} the assumption that $f(\i) = 1$ (see \cref{thm:no_restrictionC}).

\subsection{Ergodic averages along \texorpdfstring{$\Omega(m^2 + n^2)$}{Omega(m2n2)} in uniquely ergodic systems}
 
To prove Theorems \ref{thm:convergence_Gaussian} and \ref{thm:generalized_Wirsing}, instead of following the more classical approach of Wirsing, Delange, and Hal\'asz, we opt to use a dynamical approach, as developed in the recent work of Bergelson and Richter \cite{Bergelson_Richter_2020}. 
Recall that, for a natural number $n\in\N$, $\Omega(n)$ is the number of prime factors of $n$ counted with multiplicities. 
A \emph{uniquely ergodic system} is a pair $(X,T)$ where $X$ is a compact metric space, $T:X\to X$ is a continuous map and there exists a unique Borel probability measure $\mu$ on $X$ satisfying $\mu(T^{-1}A)=\mu(A)$ for every Borel set $A\subset X$. The following result was proved in \cite{Bergelson_Richter_2020}.

\begin{theorem}[{\cite[Theorem A]{Bergelson_Richter_2020}}]\label{thm:bergelson_richter}
Let $(X, T)$ be a uniquely ergodic system with the unique invariant measure $\mu$. For any $x \in X$ and any continuous function $f: X \to \C$,
\[
    \lim_{N \to \infty} \E_{1 \leq n \leq N} f(T^{\Omega(n)} x) = \int_X f \d \mu.
\]
\end{theorem}

\cref{thm:bergelson_richter}, when applied to the two-points system, reduces to (an equivalent form of) the prime number theorem.
Applying it to finite systems, one recovers a theorem of Pillai \cite{Pillai-generalization-mangoldt} and Selberg \cite{Selberg-zur} stating that $\Omega(n)$ is equally distributed over all residue classes mod $q$ for all $q \in \N$. 
\cref{thm:bergelson_richter} also contains as a special case (when applying it to irrational rotations on the torus) the Erd\H{o}s-Delange Theorem \cite{Delange-arithmetic}, which complements the results above by stating that for any irrational number $\alpha$, the sequence $(\Omega(n) \alpha)_{n \in \N}$ is uniformly distributed mod $1$. 
Here, a sequence $(a(n))_{n \in \N}$ of real numbers is called \emph{uniformly distributed mod $1$} if for every interval $I \subset [0, 1)$,
\[
    \lim_{N \to \infty} \E_{1 \leq n \leq N} 1_I(a(n) \bmod 1) = |I|.
\]

Our next theorem is an analogue of \cref{thm:bergelson_richter} along sums of two squares.

\begin{Maintheorem}
\label{thm:main-Omega-m2n2}
Let $(X,T)$ be a uniquely ergodic system with the unique invariant measure $\mu$.
Then for any $x\in X$ and any $f\in C(X)$,
\[
    \lim_{N\to\infty} \E_{1 \leq m, n \leq N} f(T^{\Omega(m^2 + n^2)}x)=\int_Xf \ d\mu.
\]
\end{Maintheorem}

A sequence $(a(m,n))_{m, n \in \N}$ of two parameters is called \emph{uniformly distributed mod $1$} if for every interval $I \subset [0, 1)$,
\[
    \lim_{N \to \infty} \E_{1 \leq m, n \leq N} 1_I(a(m, n) \bmod 1) = |I|.
\]
Taking $(X, T)$ to be the rotation by $q$ points, \cref{thm:main-Omega-m2n2} implies that $\Omega(m^2 + n^2)$ is equally distributed over all residue classes mod $q$. When applied to the rotation by an irrational $\alpha$ on the torus $\R/\Z$, we have $(\Omega(m^2 + n^2) \alpha)$ is uniformly distributed mod $1$. 
By applying \cref{thm:main-Omega-m2n2} to unipotent affine transformations on tori (following the approach of Furstenberg in \cite[pages 67 -- 69]{Furstenberg81}), we obtain the following corollary:
\begin{corollary}\label{cor:weyl}
    For $Q \in \R[x]$, $Q(\Omega(m^2 + n^2))_{m, n \in \N}$ is uniformly distributed mod $1$ if and only if at least one of the coefficients of $Q$ is irrational. 
\end{corollary}

\subsection{The main theorem}
The main technical result of this paper, stated below, is an ergodic theorem involving additive averages for multiplicative actions of the Gaussian integers.
This theorem, albeit somewhat complicated to formulate, is the main ingredient in the proofs of Theorems \ref{thm:generalized_Wirsing} and \ref{thm:main-Omega-m2n2}.

\begin{Maintheorem}\label{thm_finiterank}
    Let $X$ be a compact metric space and let $\mathcal{T}$ denote the semigroup of all continuous transformations $T:X\to X$ under composition.
    Let $\tau:(\G^*,\times)\to\mathcal{T}$ be a semigroup homomorphism such that $\tau(\P)$ is finite and let $T_1,\dots,T_k\in\tau(\P)$ be the complete list of those transformations satisfying
    \begin{equation}\label{Fe}
\sum_{p\in\P:\tau(p)=T_j}\frac1{\NN(p)}=\infty.
    \end{equation}
    Suppose that there exists a unique Borel probability measure $\mu$ on $X$ such that for all $j \in [k]$ and every Borel set $A \subset X$, $\mu(T_j^{-1} A) = \mu(A)$. 
    Then for any $x \in X$, $F \in C(X)$ and dilated F{\o}lner sequence $(\Phi_N)_{N \in \N}$,
\[
\lim_{N \to \infty} \E_{n \in \Phi_N} F(\tau(n) x) = \int_X F \d \mu.
\]
\end{Maintheorem}

\cref{thm_finiterank} is a direct analogue of \cite[Theorem B]{Bergelson_Richter_2020}, which was formulated for ``finitely generated, strongly uniquely ergodic,  multiplicative dynamical systems''. 
While drawing inspiration from \cite[Theorem B]{Bergelson_Richter_2020}, the proof of \cref{thm_finiterank} contains some major differences due to the more intricate geometry inherent in $\G$, a rank two additive group, in contrast to the simpler rank one group, $\Z$.
Moreover, we must handle arbitrary dilated F\o lner sequences in $\G^*$, which are not as rigid as the sequence of intervals $\{1,\dots,N\}$ considered in \cite{Bergelson_Richter_2020} (see \cref{remark_folnersequences}).

We also emphasize that the proof of \cref{thm_finiterank} is dynamical and combinatorial in nature, and the only number theoretic input needed (besides a version of the Tur\'an-Kubilius inequality) is some control on the distribution of the primes in $\G$, which already follows from the works of Landau \cite{Landau1903} and Hecke \cite{Hecke1918}. 

\cref{thm_finiterank} involves taking additive averages on a multiplicative dynamical system and is reminiscent of the results in our previous paper \cite{Donoso-Le-Moreira-Sun-additive-averages} regarding measure preserving actions of $(\N,\times)$. 
To put this in perspective, as opposed to the main results in \cite{Donoso-Le-Moreira-Sun-additive-averages}, which dealt with arbitrary multiplicative actions, in \cref{thm_finiterank} we require the assumption that $\tau(\P)$ is finite.
On the other hand, the conclusion in \cref{thm_finiterank} is significantly stronger than the results in \cite{Donoso-Le-Moreira-Sun-additive-averages}.

\subsection{Outline of the article} 

In \cref{sec:background}, we set up notation and present some basic facts about Gaussian integers, the prime number theorem in $\G$, and define dilated F{\o}lner sequences. 
The proof of \cref{thm_finiterank} occupies \cref{sec:proof_main_theorem}. A short proof of \cref{thm:main-Omega-m2n2} is given in \cref{sec:proof_omega_m2n2} and \cref{thm:generalized_Wirsing} is proved in \cref{sec:wirsing_gaussian}.  Some natural open questions are listed in \cref{sec:open_questions}. Lastly, the appendix contains some technical estimates which are needed for the proof of \cref{thm_finiterank} and an argument showing that a random multiplicative functions taking values in $\{-1, 1\}$ almost surely contains all finite patterns of $-1$ and $1$.

\vspace{0.3em}

\textbf{Acknowledgements.}
The authors would like to thank Vitaly Bergelson and Florian Richter for helpful discussions about their paper \cite{Bergelson_Richter_2020}.

\section{Background}
\label{sec:background}

\subsection{Notation}\label{sec:notation}

The absolute value of a complex number $z$ is written $|z|$, and its argument is denoted by $\Arg(z)\in\R/(2\pi\Z)$, usually identified with $[-\pi,\pi)$.
We also use the convention that $\Arg(0)=0$.
As usual, we denote by $\Re z$ and $\Im z$ the real and imaginary parts of $z\in\C$.
We let $S^{1}:=\{z\in\mathbb{C}\colon |z|=1\}\subset\C$ be the unit circle and $\D:=\{z\in\C:|z|\leq1\}$ be the unit disk.

We denote by $\G$ the ring of Gaussian integers $\{a + b\i: a, b \in \Z\}\subset\C$, and we use $\G^{\ast}:=\G\backslash\{0\}$ for the set of non-zero Gaussian integers.
The norm of a Gaussian integer $n$ is defined by $\NN(n) = |n|^2$. 

The greatest common divisor of two non-zero Gaussian integers $m, n$ is only unique up to multiplication by a unit, so we avoid using $\gcd(m, n)$ on its own; however the norm $\NN(\gcd(m, n))$ is well defined.

Given $A\subset \C$ and $z\in \C$, we define $A \pm z=\{a \pm z:a\in A\}$, and $zA = \{za: a \in A\}$. 
Depending on $A$ and $z$, we assign different meanings to $A/z$: if $A \subset \C$ and $z \in \C \setminus \{0\}$, define $A/z = \{a/z: z \in A\}$ \emph{unless} $A \subset \G$ and $z \in \G^*$, in which case we define $A/z = \{x \in \G: xz \in A\}$. This distinction will be clear from the context.

When $A$ is a finite set, we denote by $|A|$ its cardinality.
Given a subset $A\subset\G^{\ast}$, denote by
\[
    \LL(A):=\sum_{n\in A}\frac1{\NN(n)}
\]
its \emph{logarithmic weight}. 
If $\LL(A)=\infty$ we say that $A$ is a \emph{divergent set}. 
Otherwise we say that $A$ is a \emph{convergent set}.
It is clear that every divergent set is infinite, and whenever a divergent set is partitioned into finitely many sets, at least one of them must be divergent.

For a non-empty finite set $A$ and function $f: A\to \C$, we set:
\[
    \E_{n\in A}f(n)\coloneqq \frac{1}{|A|}\sum_{n\in A} f(n).
\]
If $A\subset \G^*$, then we set:
\[
     \E^{\log}_{n\in A}f(n)\coloneqq \frac{1}{\LL(A)}\sum_{n\in A}\frac{f(n)}{\NN(n)}.
\]

Given two functions $f, g: \G^{\ast} \to \C$, we use the following notations:
\begin{itemize}
    \item $\displaystyle f(x) = O(g(x))$ or $\displaystyle f(x) \ll g(x)$ means there is a positive constant $C$ such that $|f(x)| < C g(x)$ for all $x \in \G^{\ast}$.

    \item $\displaystyle f(x) = o(g(x))$ indicates that 
    \[
    \lim_{\NN(x) \to \infty} \frac{f(x)}{g(x)} = 0.
    \]

    \item $\displaystyle f(x) \sim g(x)$ means 
    \[
    \lim_{\NN(x) \to \infty} \frac{f(x)}{g(x)} = 1.
    \]
\end{itemize}

\subsection{Distribution of Gaussian primes}
\label{sec_gaussianprimes}
A Gaussian prime is an element $p\in\G^*$ which cannot be decomposed as $p=ab$ where $a, b$ are non-unit Gaussian integers.
We use $\P$ to denote the set of Gaussian primes and $\P_1=\{p\in\P:\Re p>0,\Im p\geq0\}$ for the restriction of $\P$ to the first quadrant. 
A Gaussian integer $a + b\i$ in the first quadrant is a Gaussian prime if and only if either:
\begin{itemize}
    \item $b=0$ and $a$ is a prime in $\N$ of the form $4n + 3$, or 
    \item $b>0$ and $a^2 + b^2$ is a prime number in $\N$ (which will not be of the form $4n + 3)$. 
\end{itemize}
Note that the units of $\G$ are $\{1,\i,-1,-\i\}$ and hence the first quadrant is a natural fundamental domain for their action on $\G$. In particular, $\P$ is invariant under multiplication by $\i$.

In our proofs we make crucial use of known results about the distribution $\P$.
In analogy to the prime number theorem, Landau \cite{Landau1903} proved that

\begin{equation}
    \label{thm:Landau1903} 
    \big|\{p \in \P: \NN(p) < N\}\big| \sim \frac{N}{\log N} \text{ as } N \to \infty.
 \end{equation}
Landau's result only estimates the number of primes in a disk around the origin. 
This was later extended by Hecke \cite{Hecke1918} who showed that the number of primes in a ``slice'' (or sector) of the complex plane is proportional to its amplitude. 

\begin{theorem}[\cite{Hecke1918}, see also \cite{Rudnick-Waxman-AnglesOfPrimes}]
\label{thm:hecke}
For any interval $I \subset [0, \pi/2]$,
\[
    \frac{|\{p \in \P: \NN(p) < N \text{ and } \Arg(p) \in I\}|}{N/\log N} \to \frac{|I|}{2\pi} \text{ as } N \to \infty.
\]
\end{theorem}

As a corollary of \cref{thm:hecke} it follows that for any interval $I \subset [-\pi, \pi)$ and any $0\leq a<b$,

\begin{equation}
    \label{thm:heckeLandau}
    \frac{|\{p \in \P:  a N\leq \NN(p)\leq bN \text{ and } \Arg(p) \in I\}|}{N/\log N} \to \frac{|I|}{2\pi}(b-a) \text{ as } N \to \infty.
\end{equation}

We will use \eqref{thm:heckeLandau} to estimate the amount of primes in dilations of certain small neighborhoods of $1$.
While we could use disks centered at $1$, in view of \eqref{thm:heckeLandau} it is more convenient to use the following annulus sectors:
For each $\epsilon\in(0,1)$ define
\begin{equation}\label{eq_polarneighborhoods}
\B{\epsilon}:=\big\{z\in\C:1-\epsilon<|z|<1+\epsilon,\ \Arg(z)\in\left(-\pi\epsilon,\pi\epsilon\right)\big\}.
\end{equation}
For each $\epsilon\in(0,1)$ and $n\in\G^*$, since $\NN(n)=|n|^2$, we have
\[
n\B\epsilon\cap\G=\Big\{m\in\G:(1-\epsilon)^2\NN(n)<\NN(m)<(1+\epsilon)^2\NN(n),\ \Arg(m)\in\Arg(n)+\big(-\pi\epsilon,\pi\epsilon\big)\Big\}.
\]
Using \eqref{thm:heckeLandau}, we deduce that for every $\epsilon>0$,
\begin{equation}
    \label{eq_PNTBepsilon}
    \lim_{n\in\G^{\ast},\NN(n)\to\infty}\frac{\big|\P\cap n\B\epsilon\big|}{\NN(n)/\log\NN(n)}= \frac{2 \pi \epsilon(4 \epsilon)}{2 \pi} = 4\epsilon^2.
\end{equation}

\subsection{Dilated F{\o}lner sequences} 

\label{sec:def_dilated_Folner_seq}

Given a function $f:\Z\to\C$ it is natural to consider its Ces\`aro average $\E_{1 \leq n \leq N} f(n)$ over the initial interval $\{1,\dots,N\}$, or the average $\E_{-N \leq n \leq N} f(n)$ over the centered interval $\{-N,\dots,N\}$.
However, when given a function $f:\G\to\C$, there are several natural ways to average $f$: one can consider, for instance, averages over disks $\E_{\NN(n)<N}f(n)$ or over the squares $\E_{1\leq 
 m, n \leq N}f(m + n\i)$ or $\E_{-N \leq m, n \leq N}f(m + n\i)$.
To simultaneously address all these averaging schemes, we make use of the notion of (additive) \emph{F\o lner sequences}; these are sequences $\Phi=(\Phi_N)_{N\in\N}$ of finite subsets of $\G$ such that for every $n\in\G$,
$$\lim_{N\to\infty}\frac{\big|(\Phi_N+n)\triangle\Phi_N\big|}{|\Phi_N|}=0.$$

\begin{remark}\label{remark_folnersequences}
Another reason we work with abstract F\o lner sequences, as opposed to simply use the squares $\Phi_N:=\{n\in\G^*:0<\Re n,\Im n\leq N\}$ is that we often have to consider the average over the set $\Phi_N/a$, for some $a\in\G$. 
Unlike the situation in $\Z$, the set $\Phi_N/a$ does not necessarily equal $\Phi_M$ for some other $M$.
Therefore, even if one is interested solely in averages over squares, it is necessary to consider more general F\o lner sequences.
\end{remark}

As already mentioned in the introduction, not every F{\o}lner sequence works in our results. 
Indeed, by \cite[Theorem 1.1]{Fish-normal_Liouville}, there is a real-valued completely multiplicative function $g \colon \N \to \{-1,1\}$ such that both $g^{-1}(\{1\})$ and $g^{-1}(\{-1\})$ contain arbitrarily long intervals. 
Consider $f\colon \G^* \to \{-1, 1\}$ defined as $f(n)= g(\NN(n))$ for $n \in \G^*$. 
It follows that $f^{-1}(\{1\})$ and $f^{-1}(\{-1\})$ contain F{\o}lner sequences $(\Phi^+_N)_{N \in \N}$ and $(\Phi^-_N)_{N \in \N}$ in $\G$, respectively. Letting $\Phi_{2N} = \Phi^+_N$ and $\Phi_{2N + 1} = \Phi^-_N$, we see $\E_{n \in \Phi_N} f(n)$ does not converge. 

There are plenty of functions $f$ as above. 
One can extend Fish's theorem \cite{Fish-normal_Liouville} on random multiplicative functions from $\Z^{\ast}$ to $\G^{\ast}$ to show that most completely multiplicative functions $f: \G^{\ast} \to \{-1, 1\}$ contains all finite patterns of $-1$ and $1$. 
This is a result of independent interest, but because it is somewhat distinct from the primary content of the article, we put it in \cref{sec:counter_example_non_dilated}.

Due to the above reason, in Theorems \ref{thm:convergence_Gaussian}, \ref{thm:generalized_Wirsing}, and \ref{thm_finiterank}, we must restrict to a subclass of F{\o}lner sequences obtained from dilating an open set.

\begin{definition}
A sequence $(\Phi_N)_{N\in\N}$ of finite subsets of $\G^*$ is called a  \emph{dilated F{\o}lner sequence} if there exists a Jordan measurable\footnote{A set $U\subset\C$ is Jordan measurable if and only if it is a bounded Lebesgue measurable set and its boundary has Lebesgue measure zero. This is equivalent to the indicator function $1_U$ being Riemann integrable.} set $U\subset\C$ of positive Lebesgue measure and a sequence of positive real numbers $(k_N)_{N \in \N}$ such that $k_N \to \infty$ as $N \to \infty$ and $\Phi_N= \G^{\ast} \cap k_N U$.
\end{definition}

We emphasize that the definition does not require $U$ to contain $0$ and it is easy to check that every dilated F\o lner sequence is indeed an additive F\o lner sequence.

\begin{example}
~

    \begin{enumerate}[label=(\roman*)]
        \item Two natural dilated F{\o}lner sequences are the squares $\Phi_N:=\{n\in\G^*:0<\Re n,\Im n\leq N\}$ and the disks $\Phi_N=\{z\in\G^*:\NN(z)\leq N^2\}$.

        \item It is immediate from the definition that any subsequence of a dilated F{\o}lner sequence is a dilated F{\o}lner sequence.
        
        \item If $(\Phi_N)_{N\in\N}$ is a dilated F\o lner sequence in $\G$ and $n\in\G$, then the sets $\Phi_N=\Phi_N/n$ also form a dilated F\o lner sequence.

        \item\label{example_notdilated} The sequence of shifted disks $\Phi_N = N^2 + (N \D\cap\G)$ is not a dilated F\o lner sequence. 
    \end{enumerate}

\end{example}
    It remains an interesting question whether the F\o lner sequence in part \ref{example_notdilated} of the example satisfies the conclusion of \cref{thm_finiterank} (see also the related \cref{question_shifteddilatedFolner}).

The main property that distinguishes dilated F\o lner sequences is captured in the following lemma and, roughly speaking, states that dilated F\o lner sequences are unchanged under small multiplicative perturbations.

\begin{lemma}\label{lemma_dilatedFolner}
    If $(\Phi_N)_{N \in \N}$ is a dilated F{\o}lner sequence, then for every $\delta>0$, there exists $\epsilon>0$ such that whenever $a,b\in\G^{\ast}$ satisfy
    $b\in a\B\epsilon$, then
    $$\lim_{N\to\infty}\frac{|\Phi_N/a\triangle\Phi_N/b|}{|\Phi_N/a|}<\delta.$$
\end{lemma}
\begin{proof}

Denote by $m$ the Lebesgue measure on $\C$.
If $V\subset\C$ is a Jordan measurable set and $t$ denotes a real parameter, then $|\G^*\cap tV|/t^2\to m(V)$ as $t\to\infty$.
Indeed,
$$
       \lim_{t\to\infty}\frac{|\G^*\cap tV|}{t^2}
       =
       \lim_{t\to\infty}\frac1{t^2}\sum_{n\in G^*}1_V(n/t)
       =
       \lim_{\epsilon\to0}\epsilon^2\sum_{n\in G^*}1_V(n\epsilon).
    $$
    The last sum is a Riemann sum for $1_V$ and since $1_V$ is Riemann integrable, the limit exists and equals $\int_\C1_V\d m=m(V)$.

Now let $U\subset\C$ be a Jordan measurable set and $(k_N)$ be a sequence of real numbers such that $\Phi_N=k_NU\cap\G^*$.
Then $\Phi_N/a=\{n\in\G^*:n\in k_NU/a\}=k_N\tfrac Ua\cap G^*$.
Since the set $\tfrac Ua$ is Jordan measurable, it follows that $|\Phi_N/a|/k_N^2\to m(\tfrac Ua)=m(U)/\NN(a)$. 
Similarly,
$\Phi_N/a\cap\Phi_N/b=k_N\Big(\tfrac Ua\cap\tfrac Ub\Big)\cap G^*$, so $|\Phi_N/a\cap\Phi_N/b|/k_N^2\to m\Big(\tfrac Ua\cap\tfrac Ub\Big)=m\Big(U\cap\tfrac abU\Big)/\NN(a)$.
It follows that
$$\lim_{N\to\infty}\frac{|\Phi_N/a\cap\Phi_N/b|}{|\Phi_N/a|}=\frac{m(U\cap (a/b)U)}{m(U)}=1-o_{a/b\to1}(1).$$
\end{proof}
We use \cref{lemma_dilatedFolner} only to prove \cref{lem:xyzuu}.

\section{Proof of \texorpdfstring{\cref{thm_finiterank}}{Theorem F}}
\label{sec:proof_main_theorem}

In this section we prove \cref{thm_finiterank}. 
Throughout the section we fix a compact metric space $X$, we let ${\mathcal T}$ denote the semigroup of continuous functions $T\colon X\to X$, and we fix a semigroup homomorphism $\tau\colon (\G,\times)\to{\mathcal T}$ such that $\tau(\P)$ is finite.
We also fix the transformations $T_1,\dots,T_k\in{\mathcal T}$ such that
\begin{equation}
\label{cond_transf}
  \parbox{\dimexpr\linewidth-7em}{$\P\cap\tau^{-1}(T_i)$ is divergent for all $i$ and $\{p\in\P:\tau(p)\notin\{T_1,\dots,T_k\}\}$ is convergent.
 }
\end{equation} 
We say that a Borel probability measure $\nu$ on $X$ is an \emph{additively empirical measure} if there exists $x\in X$ and some dilated F\o lner sequence $\Phi$ such that $\nu$ is a weak$^*$ limit
$$\nu=\lim_{N\to\infty}\E_{n\in\Phi_N}\tau(n)\delta_{x},$$
where $\delta_x$ denotes the Dirac measure at the $x$, and $\tau(n)\delta_x$ denotes the pushforward of $\delta_x$ under the transformation $\tau(n)\in{\mathcal T}$.

The first step of the proof of \cref{thm_finiterank} is to reduce it to the following statement, whose proof occupies most of the section.

\begin{theorem}\label{thm_finiterankb}
    For every additively empirical measure $\nu$ and every $j\in[k]$,
    $T_j\nu=T_j^2\nu$.
\end{theorem}

\begin{proof}[Proof of \cref{thm_finiterank} assuming \cref{thm_finiterankb}]

    Let $\mu$ be the unique probability measure on $X$ invariant under all the maps $T_j$ for $j\in[k]$.
    We need to show that for each $x\in X$, $f\in C(X)$ and every dilated F\o lner sequence $(\Phi_N)_{N \in \N}$,
\[
    \lim_{N\to\infty}\E_{n\in\Phi_N}f(\tau(n)x)=\int_XF\d\mu.
\]

This is equivalent to the statement that $\mu$ is the unique additively empirical measure.
Since $\mu$ is the unique probability measure on $X$ invariant under each $T_j$, our task will be complete if we show that any additively empirical measure $\nu$ on $X$ satisfies $T_j\nu=\nu$ for every $j\in[k]$.
This would follow directly from \cref{thm_finiterankb} if each $T_j$ were invertible, but that is not necessarily true. 
Nevertheless, it is still possible to apply \cref{thm_finiterankb} and the special nature of additively empirical measures to conclude that $T_j\nu=\nu$ for every $j\in[k]$.

Fix $j\in[k]$. Let $x\in X$ and let $\Phi=(\Phi_N)_{N\in\N}$ be a dilated F\o lner sequence such that $\nu=\lim_{N\to\infty}\E_{n\in\Phi_N}\tau(n)\delta_x$.
Fix $f\in C(X)$ and $\epsilon>0$.
We will use a version of the Tur\'an-Kubilius inequality for Gaussian integers to compare the averages of $f(\tau(n)x)$ with the averages of $f(T_j\tau(n)x)$; for completeness we formulate and prove the version we need in \cref{sec_appendixA}.

Let $B\subset\P\cap\tau^{-1}(T_j)$ be a finite set such that $\LL(B)>4/\epsilon^2$ and apply \cref{2.5} with $a(n)=f(\tau(n)x)$ together with \cref{lemma_gcdestimate} to conclude that
\begin{equation}\label{eq_proofnon-invertible1}
\limsup_{N\to\infty}\left|\E_{n\in\Phi_N}f(\tau(n)x)-\E^{\log}_{p\in B}\E_{n\in\Phi_N/p}f(T_j\tau(n)x)\right|<\epsilon.
\end{equation}
On the other hand, applying
\cref{2.5} with $a(n)=f(T_j\tau(n)x)$ together with \cref{lemma_gcdestimate}, it follows that
\begin{equation}\label{eq_proofnon-invertible2}
\limsup_{N\to\infty}\left|\E_{n\in\Phi_N}f(T_j\tau(n)x)-\E^{\log}_{p\in B}\E_{n\in\Phi_N/p}f(T_j^2\tau(n)x)\right|<\epsilon.
\end{equation}
Note that, for every $p\in\P$, the sequence $(\Phi_N/p)_{N\in\N}$ is a dilated F\o lner sequence, and hence any accumulation point of the measures $\E_{n\in\Phi_N/p}\tau(n)\delta_x$ is an additively empirical measure $\tilde\nu$. 
Using \cref{thm_finiterankb} we deduce that $T_j\tilde\nu=T_j^2\tilde\nu$, which implies that
$$\forall p\in\P,\qquad \lim_{N\to\infty}\left|\E_{n\in\Phi_N/p}f(T_j\tau(n)x)-\E_{n\in\Phi_N/p}f(T_j^2\tau(n)x)\right|=0.$$
Combining this with \eqref{eq_proofnon-invertible1} and \eqref{eq_proofnon-invertible2}, we conclude that
\begin{equation*}
\limsup_{N\to\infty}\left|\E_{n\in\Phi_N}f(\tau(n)x)-\E_{n\in\Phi_N}f(T_j\tau(n)x)\right|<2\epsilon.
\end{equation*}
Taking $\epsilon\to0$ it follows that
$$\int_Xf\d\nu
=
\lim_{N\to\infty}\E_{n\in\Phi_N}f(\tau(n)x)
=
\lim_{N\to\infty}\E_{n\in\Phi_N}f(T_j\tau(n)x)
=
\int_Xf\d(T_j\nu).$$
Since $f$ was arbitrary we conclude that $\nu=T_j\nu$, and this finishes the proof.
\end{proof}

\subsection{Roadmap of the proof of \texorpdfstring{\cref{thm_finiterankb}}{Theorem 3.1}}
Our main (and only) tool to show that $T\nu=R\nu$ for some transformations $T,R\in{\mathcal T}$ is the following result which is inspired by \cite[Proposition 2.1]{Bergelson_Richter_2020}.

\begin{lemma}\label{lemma_cleanedTK}
    Suppose that $T,R\in {\mathcal T}$ commute and for every $\epsilon>0$ there are a nonempty finite set $F\subset \G^{\ast}$  and an injective map $\alpha:F\to \G^{\ast}$ such that
    \begin{itemize}
        \item $\alpha(n)\in n\B\epsilon$ for all $n\in F$, where $\B\epsilon$ is defined in \eqref{eq_polarneighborhoods},
        \item $\tau(F)=\{T\}$ and $\tau(\alpha(F))=\{R\}$.
        \item $\displaystyle\E_{n,m\in F}^{\log}\NN(\gcd(n,m))<1+\epsilon$ and $\displaystyle\E_{n,m\in \alpha(F)}^{\log}\NN(\gcd(n,m))<1+\epsilon$.
    \end{itemize} 
    Then $T\nu=R\nu$ for every empirical measure $\nu$.
\end{lemma}

This lemma applies, for instance, if we can find an arbitrarily large set $F$ of primes contained in $\tau^{-1}(T)$ such that for each element $n \in F$ there is a nearby prime $\alpha(n)$ which is contained in $\tau^{-1}(S)$.
For our applications we are not able to ensure the existence of a such a set $F$ contained in the primes, and so the presence of the third bullet point insures that we can still invoke a variant of the Tur\'an-Kubilius inequality (\cref{2.5}) to obtain the desired conclusion. 
The idea of the proof of \cref{lemma_cleanedTK} can be summarized by the following chain of approximations,
    \begin{align*}
        \int f\d(T\nu)\approx\E_{n \in \Phi_N} f(\tau(n) Tx) & \approx \E^{\log}_{p \in \alpha(F)} \E_{n \in \Phi_N/p} f(\tau(n) TR x)  \\ & \approx \E_{p \in F}^{\log} \E_{n \in \Phi_N/p} f(\tau(n) TR x) \approx \E_{n \in \Phi_N} f(\tau(n) Rx)\approx\int f\d(R\nu),
    \end{align*}
 the formal versions of which are proven in \cref{subsection_lemma_cleanedTK} below.

Unfortunately, \cref{lemma_cleanedTK} requires strong assumptions, and in particular we can not use it directly with $(T,R)=(T_j,T_j^2)$ as the hypothesis are not met in general.
We will instead use \cref{lemma_cleanedTK}  to show that for every $j$ there exists $i$ such that $T_j\nu=T_i\nu$ and also $T_j^2\nu=T_i\nu$.

In view of the first condition in \cref{lemma_cleanedTK} it is convenient to introduce the following notion:

\begin{definition}
    Let $D\subset \G^{\ast}$ and $\epsilon>0$. A map $\alpha:D\to \G^{\ast}$ is an \emph{$\epsilon$-map} if $\alpha(d)\in d\B\epsilon$ for every $d\in D$.
\end{definition}

In the following lemma we find, for each $j\in[k]$, some $i\in[k]$ (which might be equal to $j$) satisfying some technical property. 
In the sequel we show that whenever $j$ and $i$ satisfy this technical property, then both $T_j\nu=T_i\nu$ and $T_j^2\nu=T_i\nu$ for every empirical measure $\nu$.
\begin{lemma}\label{lemma_findingi}
    Let $(X,\tau)$ and $T_1,\dots,T_k$ be as in \eqref{cond_transf}. 
    Then for every $j\in[k]$ there exists $i\in[k]$ satisfying the following property. 
\begin{equation}\tag{$\star$}\label{eq_star}
\begin{split}
    &\text{For every $\epsilon>0$ there exist subsets $F,D\subset\P\cap\tau^{-1}(T_j)$, where $F$ is finite with}\\
    &\text{$\LL(F)>1/\epsilon$ and $D$ is divergent, and an injective $\epsilon$-map $\alpha:FD\to\P\cap\tau^{-1}(T_i)$.}
    \end{split}
\end{equation}
\end{lemma}

The next lemma shows that if $i$ and $j$ satisfy the property \eqref{eq_star} described in \cref{lemma_findingi}, then $T_i\nu=T_j\nu$ for every empirical measure $\nu$. 
\begin{lemma}\label{lemma_closingtheequivalencerelation}
Let $(X,\tau)$ and $T_1,\dots,T_k$ be as in \eqref{cond_transf}, and let $\nu$ be an empirical measure.
Suppose $i,j\in[k]$ and for every $\epsilon>0$ there exist a divergent set $D\subset\P\cap\tau^{-1}(T_j)$, some number $m\in\G^{\ast}$ and an injective $\epsilon$-map $\alpha:mD\to\P\cap\tau^{-1}(T_i)$.
Then $T_i\nu=T_j\nu$.
\end{lemma}

Finally, the next lemma shows that whenever $i$ and $j$ satisfy the property \eqref{eq_star} described in \cref{lemma_findingi}, then $T_i\nu=T_j^2\nu$ for every empirical measure $\nu$. 

\begin{lemma}\label{thm_combinatorialreduction}
Let $(X,\tau)$ and $T_1,\dots,T_k$ be as in \eqref{cond_transf}, and let $\nu$ be an empirical measure.
If $i,j\in[k]$ satisfy \eqref{eq_star}, then $T_i\nu=T_j^2\nu$.
\end{lemma}

The proofs of the various lemmas above are postponed to future subsections. 
For now we verify that together they imply \cref{thm_finiterankb}.
\begin{proof}[Proof of \cref{thm_finiterankb}]
    Given $j\in[k]$, use \cref{lemma_findingi} to find $i\in[k]$ so that property \eqref{eq_star} holds.
    Then the conditions of \cref{lemma_closingtheequivalencerelation} hold (taking an arbitrary $m\in F$) and hence $T_i\nu=T_j\nu$.
    Finally, \cref{thm_combinatorialreduction} implies that $T_i\nu=T_j^2\nu$ and hence we conclude that $T_j\nu=T_i\nu=T_j^2\nu$, finishing the proof.
\end{proof}

\subsection{Proof of \texorpdfstring{\cref{lemma_cleanedTK}}{Lemma 3.2}}
\label{subsection_lemma_cleanedTK}

 First, we need a technical lemma, which makes use of the fact that the F\o lner sequences we consider are dilated. 

\begin{lemma}\label{lem:xyzuu}
Let $\delta \in (0, 1)$ and let $\Phi=(\Phi_N)_{N\in\N}$ be a dilated F{\o}lner sequence. 
Then there exists $\epsilon\in(0,\delta)$ such that for every finite subset $F$ of $\G^*$, every $\epsilon$-map $\alpha: F \to \G^{\ast}$ and every function $a: \G^{\ast}\to \C$ bounded by $1$, we have
\[
    \limsup_{N \to \infty} \left| \E^{\log}_{p \in F} \E_{n \in \Phi_N/p} a(n) - \E^{\log}_{q \in \alpha(F)} \E_{n \in \Phi_N/q} a(n)  \right| \leq \delta.
\]
\end{lemma}

\begin{proof}
Using \cref{lemma_dilatedFolner}, there exists $\epsilon>0$ such that
\begin{equation*}
    \label{eq:pqf}
     \limsup_{N \to \infty}\frac{\Big| \big(\Phi_N/p\big)\triangle\big(\Phi_N/q\big)\Big|}{\big|\Phi_N/p\big|} \leq \frac{\delta}{10}
\end{equation*}
whenever $q\in p\B\epsilon$.

Since $\Big|\big|\Phi_N/p\big|-\big|\Phi_N/q\big|\Big|\leq\Big| \big(\Phi_N/p\big)\triangle\big(\Phi_N/q\big)\Big|$, this implies that 
\begin{equation*}\label{eq:pqf_2}
    \limsup_{N \to \infty}\frac{\Big| \big(\Phi_N/p\big)\triangle\big(\Phi_N/q\big)\Big|}{\big|\Phi_N/q\big|} \leq \frac{\delta}{9}.
\end{equation*}
Now if $\alpha:F\to\G^{\ast}$ is an $\epsilon$-map and $a:\G^{\ast}\to\D$ then, for every $p\in F$, it is not hard to see that
\begin{equation*}
     \limsup_{N \to \infty}\left| \E_{n \in \Phi_N/p} a(n) - \E_{n \in \Phi_N/\alpha(p)} a(n) \right| \leq \frac\delta3. 
\end{equation*}
The conclusion now follows directly from \cref{lemma_estimateaverages} with $w(p)=1/\NN(p)$, $v(p)=1/\NN(\alpha(p))$, $f(p)=\E_{n \in \Phi_N/p} a(n)$ and $g(p)=\E_{n \in \Phi_N/\alpha(p)} a(n)$.
\end{proof}
The proof of \cref{lemma_cleanedTK} uses a version of the Tur\'an-Kubiliys inequality for Gaussian integers. 
The precise version we need and its proof are provided in \cref{sec_appendixA} for completeness.

\begin{proof}[Proof of \cref{lemma_cleanedTK}]
Let $T,R\in {\mathcal T}$ be two commuting transformations satisfying the hypothesis of the lemma.
Let $\nu$ be an arbitrary additively empirical measure and let $f\in C(X)$ with $\|f\|_\infty\leq1$. 
We need to show that
\[
    \int_X f\circ T\d\nu = \int_X f \circ R \d\nu.
\]
Since $\nu$ is an additively empirical measure, it suffices to show that for every $x \in X$, every dilated F{\o}lner sequence $(\Phi_N)_{N\in\N}$ and every $\delta>0$,
\begin{equation}\label{eq_proof_lemma_turankubiliuscorollary}
    \limsup_{N\to\infty}\left|\E_{n\in\Phi_N} f(\tau(n) T x) - \E_{n\in\Phi_N} f(\tau(n) R x) \right| < 4\delta.
\end{equation}
Fix $\delta>0$ and let $\epsilon\in(0,\delta^{2})$ be given by \cref{lem:xyzuu}.
Let $F \subset \tau^{-1}(T)$ and $\alpha\colon F \to \G^{\ast}$ satisfy the lemma's hypothesis. 

Note that if $p \in \alpha(F) \subset \tau^{-1}(R)$, then for any $n \in \G^*$, $\tau(np) = \tau(n) R$. 
Therefore, by applying the Tur\'an-Kublius inequality (\cref{2.5}) with $a(n)=f(\tau(n) T x)$ and $B = \alpha(F)$, we have
\begin{equation}\label{eq_proof_lemma_turankubiliuscorollary1}
\limsup_{N \to \infty} \left|\E_{n\in\Phi_N} f(\tau(n) T x) - \E_{p\in \alpha(F)}^{\log}\E_{n \in \Phi_N/p}f(\tau(n) TR x)\right|^2\leq\E_{p,q\in \alpha(F)}^{\log} \NN(\gcd(p,q)) - 1\leq \epsilon\leq \delta^{2}
\end{equation}
Similarly, by using \cref{2.5} with $a(n)=f(\tau(n) R x)$ and $B = F$, we get 
\begin{equation}\label{eq_proof_lemma_turankubiliuscorollary2}
    \limsup_{N \to \infty} \left|\E_{n\in\Phi_N}f(\tau(n) Rx) - \E_{p\in F}^{\log}\E_{n\in\Phi_N/p}f(\tau(n)TR x)\right|^2\leq\E_{p,q\in F}^{\log} \NN(\gcd(p,q)) - 1 \leq \epsilon\leq \delta^{2}.
\end{equation}
By \cref{lem:xyzuu}, we have
\begin{equation}\label{eq:psi/p}
    \limsup_{N \to \infty} \left| \E_{p\in \alpha(F)}^{\log}\E_{n \in \Phi_N/p}f(\tau(n) TR x) - \E_{p\in F}^{\log}\E_{n\in\Phi_N/p}f(\tau(n)TR x)\right| \leq 2\delta.
\end{equation}
Relations \eqref{eq_proof_lemma_turankubiliuscorollary1}, \eqref{eq_proof_lemma_turankubiliuscorollary2}, and \eqref{eq:psi/p} give \eqref{eq_proof_lemma_turankubiliuscorollary}, finishing the proof.
\end{proof}

\subsection{Proof of \texorpdfstring{\cref{thm_combinatorialreduction}}{Lemma 3.6}}

We will use \cref{lemma_cleanedTK} directly with $T=T_j^2$ and $R=T_i$.

Let $\epsilon>0$. Without loss of generality we can assume that $\epsilon<1$. 
Using \eqref{eq_star}, we can find sets $F,D\subset\P\cap\tau^{-1}(T_j)$ such that $F$ is finite with $\log(F)>1/\epsilon$, $D$ is divergent, and there is an injective $\epsilon$-map $\alpha:FD\to\P\cap\tau^{-1}(T_i)$.
Since $F$ is finite and $D$ is divergent, we can find a finite set
  $F_1\subset D$ which is disjoint from $F$ and satisfies $\log(F_1)>1/\epsilon$. Let $F_{2}=FF_1$.
The restriction of $\alpha$ to $F_{2}$ is an $\epsilon$-map, and by construction $\tau(F_{2})=\{T_j^2\}$ and $\tau(\alpha(F_{2}))=\{T_i\}$.

One can compute that
$$\log(\alpha(F_{2}))=\sum_{n\in F_{2}}\frac1{\NN(\alpha(n))}\geq\sum_{n\in F_{2}}\frac1{4\NN(n)}=\frac{\log(F_{2})}{4}=\frac{\log(F)\log(F_1)}{4}\geq\frac1{4\epsilon^2}.$$
Finally, using \cref{lemma_gcdestimate}, we deduce that 
$$\E^{\log}_{n,m\in F_{2}}\gcd(n,m)<(1+8\epsilon)^2<1+80\epsilon,\qquad\text{ and }\qquad\E^{\log}_{n,m\in\alpha(F_{2})}\gcd(n,m)<1+16\epsilon^2.$$
Since $\epsilon$ can be taken arbitrarily small, we meet the conditions of \cref{lemma_cleanedTK}, and hence conclude that $T_j^2\nu=T_i\nu$ as desired.

\subsection{Sparse sets}\label{subsection_coloring}
By assumption, when restricted to primes, $\tau$ takes only finitely many values. 
Moreover, $\tau(p)\in\{T_1,\dots,T_k\}$ for all primes $p$ outside a convergent set.
Throughout the proof we also need to deal with numbers that are not primes, and we seek to find injective $\epsilon$-maps into sets of primes.
For that it is important to know the value of $\tau$ on the primes ``near'' a given $n$.
For this purpose, it is convenient to introduce the following colorings of $\G^*$:
for each $\delta>0$ we define $\chi=\chi_\delta:\G^*\to\{0,1,\dots,k\}$ by letting $\chi(n)$ be an arbitrary $\ell\in[k]$ satisfying 
$$\big|\P\cap n\B\delta\cap\tau^{-1}(T_\ell)\big|\geq\frac{3.5\delta^2\NN(n)}{k\log\NN(n)}.$$
If no such $\ell\in[k]$ exists, we let $\chi(n)=0$. 
Note that, in view of the prime number theorem \eqref{eq_PNTBepsilon}, for each $\delta>0$ only finitely many $n\in\G^*$ get colored $0$.

We want to build injective $\delta$-maps from certain monochromatic divergent sets $D\subset\chi^{-1}(\ell)$ to the set $\P\cap\tau^{-1}(T_\ell)$. 
To get injectivity we need to make sure that the elements of $D$ are not too clumped together; indeed some divergent sets have no divergent subset with an injective $\delta$-map into the primes (see \cref{example_sparsityneeded} below).
This forces us to introduce the notion of sparse sets:

\begin{definition}
For $\delta>0$, we say an infinite set $D \subset \G$ is \emph{$\delta$-sparse} if 
\begin{equation}\label{eq_lemma_selectingF4temp}
|D\cap n\B{3\delta}|<\frac{\NN(n)}{\log\NN(n)}\left(\frac{3\delta^2}{k}+o_{\NN(n)\to\infty}(1)\right).
\end{equation}
\end{definition}

In the next lemma, sparcity is used to find $\delta$-maps into sets of primes.

\begin{lemma}\label{lemma_coloredprimes3}
Let $\delta>0$ and suppose $D\subset\G^{\ast}$ is an infinite $\delta$-sparse set such that $\chi_{\delta}(D)=\{j\}$ for some $j\in[k]$. 
Then there exists a co-finite subset $D'\subset D$ and an injective $\delta$-map $\alpha:D' \to \P\cap\tau^{-1}(T_j)$.
\end{lemma}

\begin{proof}
Since $D$ is $\delta$-sparse, there exists a co-finite subset $D' \subset D$ such that 
\begin{equation}\label{eq:FminusF'}
    |D'\cap n\B{3\delta}|<\frac{3.5 \delta^2 \NN(n)}{k\log\NN(n)} \text{ for every } n \in \G^*.
\end{equation}
For convenience denote $\chi:=\chi_{\delta}$ and
enumerate $D'=\{n_1,n_{2},\dots\}$. 
Since $\chi(n_1)=j$, there is some prime $p_1 \in \P\cap\tau^{-1}(T_{j})$ such that $p_1 \in n_1\B{\delta}$. 
Let $\alpha(n_1) = p_1$.
Continuing in this manner, we seek to define recursively, for each $i>1$, $\alpha(n_i)$ to be any prime $p_i \in \P\cap\tau^{-1}(T_{j})\setminus\{\alpha(n_1),\dots,\alpha(n_{i-1})\}$ with $p_i \in n_i\B{\delta}$. 
We claim that such $p_i$ exists. 
Since $\chi(n_i)=j$, there are at least $\frac{3.5 \delta^2\NN(n_i)}{k\log\NN(n_i)}$ many primes $p$ in $\tau^{-1}(T_{j}) \cap n_i \B{\delta}$.
Let 
\[
    F_i:=\big\{n_t:t<i\text{ and }\alpha(n_t) \in n_i \B{\delta}\big\}.
\]
For each $n_t \in F_i$, we have $\alpha(n_t)\in n_t\B{\delta}\cap n_i\B{\delta}$, and so this intersection is non-empty. 
Thus,
\[
    n_t\in 
\alpha(n_t)\B{\delta}^{-1}\subset n_i\B{\delta}\B{\delta}^{-1}\subset n_{i} S_{3\delta}.
\]
It follows that 
$F_i\subset D' \cap n_i S_{3\delta}$. 
Since $D'$ satisfies \eqref{eq:FminusF'}, we have $|F_i|<\frac{3.5 \delta^2\NN(n_i)}{k\log\NN(n_i)}$. Hence, there must be a prime $p$ other than $\alpha(n_1),\dots,\alpha(n_{i-1})$ belonging to $\tau^{-1}(T_j) \cap n_i\B{\delta}$.

It is not hard to see that $\alpha$ is an injective $\d$-map on $D'$.
\end{proof}

Unfortunately, it is not true that every divergent set contains a sparse divergent subset, as shown by the next example.

\begin{example}\label{example_sparsityneeded}
    Fix $\delta>0$, $A\subset\G^{\ast}$, and let 
    $$E_A:=\bigcup_{n\in A}\big(\G\cap n\B\delta\big).$$
    It is not difficult to compute that $\log\big(\G\cap n\B\delta\big)\approx\delta^2$.
    Therefore, if the sets $n\B\delta$ with $n\in A$ are pairwise disjoint then $E_A$ is divergent as long as $A$ is infinite.
    On the other hand, if $S$ is a $\delta$-sparse set, then $$\log(S\cap n\B\delta)\ll_\delta \frac1{\NN(n)}\big|S\cap n\B\delta\big|\ll_\delta \frac1{k\log\NN(n)},$$
    so if $\sum_{n \in A} \frac{1}{\log\NN(n)}<\infty$, then any $\delta$-sparse subset of $E_A$ must be convergent.
\end{example}

However, the next lemma shows that any divergent set of primes has a sparse divergent subset. In fact, we need something stronger.

\begin{lemma}\label{lemma_tobeproved5}
Let $D\subset\P$ be a divergent set, let $F\subset\G^*$ be a finite set and let $\delta\in(0,1/3)$.
Suppose that 
\begin{equation}\label{eq_Fissparse}
\forall n,m\in F,\qquad n\in m\B{10\delta}\quad\Rightarrow\quad n=m.
\end{equation}
Then there exists a divergent subset $D'\subset D$ such that the product set $FD'$ satisfies
\begin{equation}\label{eq_FF'issparse}
\forall n\in\G^{\ast},\qquad|FD'\cap n\B{3\delta}|<\frac{3\delta^2\NN(n)}{k\log\NN(n)}.
\end{equation}
In particular, $FD'$ is $\delta$-sparse.
\end{lemma}

\begin{proof}
Consider the set $S_1:=(F\cup\{1\})\B{3\delta}$ and let $S_2=S_1S_1^{-1}$ and $S=S_2S_2^{-1}$.
Before we proceed with the proof we make some observations that will be useful.
Note that $1\in S_1$ so $S_1\subset S_2\subset S$.
The set $F$ is finite, so $C_0:=\sup\{|z|^2:z\in S\}$ is a positive and finite number that depends only on $\delta$ and $F$.
The inverse $S^{-1}$ equals $S$, so for $n,m\in\G$ the statements $n\in mS$ and $m\in nS$ are equivalent.
Since $S$ is bounded, it follows from \eqref{thm:Landau1903} that 
\begin{equation*}\label{eq_lemma_tobeproved5_1}
    \forall n\in\G^{\ast},\qquad\big|\P\cap nS\big|\leq C_1\frac{\NN(n)}{\log\NN(n)}
\end{equation*}
for some constant $C_1>0$ that depends only on $\delta$ and $F$.
Moreover, for any $m,n\in\G$ with $m\in nS$, we have $\NN(m)\geq \NN(n)/C_0$.
So
\begin{equation*}\label{eq_lemma_tobeproved5_2}
    \forall n\in\G^{\ast},\qquad\log\big(\P\cap nS\big)=\sum_{p\in\P\cap nS}\frac1{\NN(p)}\leq\frac{\big|\P\cap nS\big|}{\NN(n)/C_0}\leq \frac{C_0C_1}{\log\NN(n)}.
\end{equation*}
Therefore, for any divergent set $D_0\subset\P$, any subset $D_1\subset D_0$ satisfying $D_0\subset D_1S$
must satisfy $\sum_{d\in D_1}\frac1{\log\NN(d)}=\infty$.

We are now ready to proceed with the proof.
Construct $D'$ using a greedy algorithm as follows:
Enumerate the elements of $D=\{d_1,d_2,\dots\}$ so that $\NN(d_i)\leq\NN(d_{i+1})$.
Let $D_0'=\emptyset$ and for each $i\in\N$ let $D_i':=D_{i-1}'\cup\{d_i\}$ if it satisfies \eqref{eq_FF'issparse}, or $D_i':=D_{i-1}'$ otherwise.
Take $D'=\bigcup D_i'$.
Note that $D'$ satisfies \eqref{eq_FF'issparse}, but for any $d\in D\setminus D'$, the set $D'\cup\{d\}$ does not. 
We will show that the set $D'$ thus constructed is divergent, and this will finish the proof.

Let $D_0=D\setminus D'$.
If $D_0$ is not divergent then $D'$ must be divergent, finishing the proof, so we will now assume that $D_0$ is a divergent subset of $\P$.
Using again a greedy algorithm, we may find a maximal subset $D_1\subset D_0$ satisfying
\begin{equation}\label{eq_lemma_tobeproved5_3}
\forall n,m\in D_1,\qquad n\in mS\quad\Rightarrow\quad n=m.
\end{equation}
Since $D_1$ is maximal we have $D_0\subset D_1S$, so it follows from the observations at the beginning of the proof that $\sum_{d\in D_1}\frac1{\log\NN(d)}=\infty$.

We claim that for each $d\in D_1$, $\log(D'\cap dS_2)\geq \frac{C_2}{\log\NN(d)}$, where $C_2>0$ is a constant that only depends on $\delta$, $F$ and $k$.
Note that \eqref{eq_lemma_tobeproved5_3} and the fact that $S=S_2S_2^{-1}$ imply that the sets $dS_2$ with $d\in D_1$ are pairwise disjoint.
Therefore if we prove this claim, then we finish the proof that $\log(D')=\infty$.

Since $d\in D_1\subset D\setminus D'$ and $D'$ is a maximal subset of $D$ satisfying \eqref{eq_FF'issparse}, it follows that $\tilde D:=D'\cup\{d\}$ does not satisfy \eqref{eq_FF'issparse} and hence there exists $n\in\G^*$ such that the set
$A:=F\tilde D\cap n\B{3\delta}$ satisfies  
$$|A|\geq \frac{3\delta^2\NN(n)}{k\log\NN(n)}.$$

Observe that necessarily $n\in Fd\B{3\delta}^{-1}$; indeed since $D'$ does satisfy \eqref{eq_FF'issparse}, there exists $a_0\in A\cap Fd$.
Since $A\subset n\B{3\delta}$, it follows that $n\in a_0\B{3\delta}^{-1}\subset Fd\B{3\delta}^{-1}$.

For each $a\in A$, there exists $\tilde a\in\tilde D$ such that $a\in F\tilde a$. 
In view of \eqref{eq_Fissparse} (and the fact that $A\subset n\B{3\delta}$), the map $a\mapsto\tilde a$ is injective, so $a_0$ is the only element of $A$ with $\tilde{a_0}=d$.
Letting $\tilde A:=\big\{\tilde a:a\in A\big\}\setminus\{d\}$ we have $|\tilde A|=|A|-1$ and $\tilde A\subset D'$.
Moreover, 
$$\tilde A
\subset 
A F^{-1}
\subset 
n\B{3\delta}F^{-1}
\subset 
Fd\B{3\delta}^{-1}\B{3\delta}F^{-1}
\subset 
d S_2,$$
so $\tilde A\subset D'\cap dS_2$ and we've reduced the claim to showing that $\log(\tilde A)\geq\frac{C_2}{\log\NN(d)}$.
By removing from $D$ a finite set (depending only on $\delta,k$ and $F$), we may assume that $\NN(n)$ is large enough so that $|A|>2$, and hence $|\tilde A|\geq|A|/2$.
Also, since $\tilde A\subset n\B{3\delta}F^{-1}\subset nS$, it follows that for every $\tilde a\in\tilde A$, we have the estimate $\NN(\tilde a)\leq\NN(n)C_0$.
Similarly, from $n\in Fd\B{3\delta}^{-1}\subset dS$ it follows that $\NN(n)\leq\NN(d)C_0$.
By further removing from  $D$ those $d$ with $\NN(d)\leq C_{0}$  (which is a finite set that depends only on $\delta,k$ and $F$), we may assume that $\NN(n)\leq\NN(d)^2$.
We conclude that
\begin{equation}
\log(\tilde A)
=
\sum_{\tilde a\in \tilde A}\frac1{\NN(\tilde a)}
\geq
\frac{|\tilde A|}{\max\{\NN(\tilde a):\tilde a\in \tilde A\}}
\geq
\frac{|A|}{2C_0\NN(n)}
\geq 
\frac{3\delta^2}{2C_0k\log\NN(n)}
\geq 
\frac{3\delta^2}{4C_0k\log\NN(d)}.
\end{equation}
This proves the claim with $C_2=\frac{3\delta^2}{4C_0k}$.
\end{proof}

Applying \cref{lemma_tobeproved5} for the set $F = \{1\}$, we deduced the following:

\begin{corollary}\label{lemma_sparserties}
For every $\delta>0$, any divergent set of $\P$ has a $\delta$-sparse divergent subset.
\end{corollary}

\subsection{Proof of \texorpdfstring{\cref{lemma_findingi}}{Lemma 3.4}}

Let $(X,\tau),T_1,\dots,T_k$ and $j\in[k]$ be as in the statement of the lemma.
Let $I\subset[k]$ be the set of those $i\in[k]$ for which \eqref{eq_star} does not hold. 
We want to show that $I$ is not all of $[k]$.
For each $\ell\in I$ there is some $\epsilon_\ell$ for which \eqref{eq_star} fails.
Take $\epsilon>0$ smaller than all the $\epsilon_\ell$.
We will find some $i\in[k]$ for which \eqref{eq_star} holds with this choice of $\epsilon$.
By construction, such $i$ will not be in $I$, finishing the proof.

With the choice of $\epsilon$ from the previous paragraph, take a finite set $\tilde F\subset\P\cap\tau^{-1}(T_j)$ with $\log(\tilde F)>k/\epsilon$.
Pick $\delta>0$ sufficiently small so that the ratio between any two distinct elements of $\tilde F$ is outside $\B{10\delta}$. 
We may also require that $\delta<\min\{1/3,\epsilon\}$.
Consider the coloring $\chi=\chi_\delta:\G^*\to\{0,\dots,k\}$ described in \cref{subsection_coloring}.

Recall that $\chi(n)=0$ for only finitely many $n\in\G^*$.
Therefore, for each sufficiently large $n\in\P\cap\tau^{-1}(T_j)$, we may 
find a subset $F_n\subset\tilde F$ with $\log(F_n)>1/\epsilon$ and such that $nF_n$ is monochromatic with a color in $[k]$.  (Here, we partition $\tilde{F}$ as $\bigcup_{1\leq i\leq k}\{p\in \tilde{F}: \chi(np)=i\}$ in order to find $F_n$.)
Since there are only finitely many possible subsets of $\tilde F$, we can find a subset $F\subset\tilde F$ with $\log(F)>1/\epsilon$ and a divergent set $D_1\subset\P\cap\tau^{-1}(T_j)$ such that $F_n=F$ for every $n\in D_1$.
We may then find a color $i\in[k]$ and a divergent subset $D_2\subset D_1$ such that $\chi(Fd)=\{i\}$ for all $d\in D_2$.
In other words, $\chi(FD_2)=\{i\}$ for this choice of $D_2$.

We can now apply \cref{lemma_tobeproved5} to find a subset $D\subset D_2$ which is divergent and such that $FD$ is $\delta$-sparse.
Finally, we apply \cref{lemma_coloredprimes3} to find a $\delta$-map $\alpha:FD'\to\P\cap\tau^{-1}(T_i)$ for some co-finite subset $D'$ of $D$.
We conclude that \eqref{eq_star} holds with the choice of $\epsilon$ in the first paragraph, which implies that $i\notin I$ and hence finishes the proof.

\subsection{Proof of \texorpdfstring{\cref{lemma_closingtheequivalencerelation}}{Lemma 3.5}}

In this section, we fix a function $\lfloor\cdot\rfloor:\C\to\G$ satisfying $\big|\lfloor z\rfloor-z\big| < 1$ for all $z\in\C$. Given a set $S\subset\C$, we denote by $\lfloor S\rfloor$ the set $\{\lfloor s\rfloor:s\in S\}$.

\begin{lemma}\label{lem:r_delta_sparse}
Let $\delta\in(0,1)$ and let $D \subset \G$ be a $\delta$-sparse set. If $z\in\C$ with $|z|\geq1$, then $\lfloor zD\rfloor$ is also $\delta$-sparse.
\end{lemma}

\begin{proof}
    For each $n\in\G^{\ast}$, note that 
    $$\big\{u\in\C:\lfloor u\rfloor\in n \B{3\delta}\big\}\subset  n \B{3\delta}+\D,$$
    and hence for all $d\in D$ with $\lfloor z d \rfloor \in n S_{3 \delta}$, we have that
    $$d\in \frac nz\B{3\delta}+\frac{1}{|z|}\D\subset \left\lfloor\frac nz\right\rfloor\B{3\delta}+5\D.$$
    Since the boundary of $\B{3\delta}$ has length $O(\delta)$, the number of lattice points in $\big(a\B{3\delta}+5\D\big)\setminus a\B{3\epsilon}$ is $O(|a|\delta)$ and hence we have
    \begin{eqnarray*}
        \big|\lfloor z D \rfloor \cap n S_{3 \delta}\big|
        &\leq&  
        \left|D  \cap \left(\left\lfloor \frac nz\right\rfloor S_{3 \delta}+3\D\right)\right|
        \\&\leq&
        \left|D  \cap \left(\left\lfloor \frac nz\right\rfloor S_{3 \delta}\right)\right|+O\left(\left|\left\lfloor \frac nz\right\rfloor\right|\right)
        \leq
        \frac{\NN(n)}{\log \NN(n)}\left(\frac{3\delta^2}k+o_{\NN(n)\to\infty}(1)\right).
    \end{eqnarray*}
\end{proof}

\begin{proof}[Proof of \cref{lemma_closingtheequivalencerelation}]

    Let $i,j\in [k]$ satisfy the hypothesis of \cref{lemma_closingtheequivalencerelation}.
    Fix an empirical measure $\nu$ and define $K=\{\ell\in[k]:T_\ell\nu=T_j\nu\}$; we want to show that $i\in K$.
    For each pair $(\ell,\ell')\in[k]^2$ with $\ell\in K$ and $\ell'\notin K$ we have $T_\ell\nu = T_j \nu \neq T_{\ell'}\nu$. 
    Therefore, in view of \cref{lemma_cleanedTK}, for each such pair $(\ell,\ell')$, there exists some $\epsilon_{\ell,\ell'}>0$ for which the conditions in \cref{lemma_cleanedTK} do not hold. 
    More precisely, it is not possible to find a finite set $F \subset \tau^{-1}(T_{\ell})$ and an injective $\epsilon_{\ell, \ell'}$-map $\alpha: F \to \tau^{-1}(T_{\ell'})$ that satisfy 
    \begin{equation}\label{eq:F_epsilon_fact}
        \E_{n,m\in F}^{\log}\NN(\gcd(n,m))<1+\epsilon_{\ell,\ell'}\qquad \text{ and } \qquad\E_{n,m\in \alpha(F)}^{\log}\NN(\gcd(n,m))<1+\epsilon_{\ell,\ell'}.   
    \end{equation}
    Let
    $$\epsilon=\min\{\epsilon_{\ell,\ell'}:\ell\in K,\ell'\in[k]\setminus K\}.$$
    Using the fact that there are only finitely many such pairs $(\ell,\ell')$, we deduce that $\epsilon>0$.

On the other hand, part (i) of \cref{lemma_gcdestimate} states that for any finite set $F \subset \P$,
\[
    \E_{n,m\in F}^{\log}\NN(\gcd(n,m)) < 1 + \frac{4}{\log(F)}.
\]
As a result, if $D$ is a divergent subset of $\P$ and $\alpha: D \to \P$ is an $\epsilon$-map, then there is a finite set $F \subset D$ such that \eqref{eq:F_epsilon_fact} holds.
This fact together with the definition of $\epsilon$ imply that
    \begin{equation}
  \tag{P}\label{quote}
  \parbox{\dimexpr\linewidth-7em}{
     Whenever $\ell\in K$ and $\ell'\in[k]$, if there is an injective $\epsilon$-map between a divergent subset of   $\P\cap\tau^{-1}(T_{\ell})$ and $\P\cap\tau^{-1}(T_{\ell'})$, then also $\ell'\in K$.
  }
\end{equation}

Recall that our aim is to prove that $i\in K$, while trivially $j\in K$. 
Unfortunately we cannot simply invoke \eqref{quote} as we do not have a-priori an $\epsilon$-map from a divergent subset of $\P \cap \tau^{-1}(T_j)$ to $\P \cap \tau^{-1}(T_i)$. 
However, the assumptions of the lemma provide us with a divergent set $D \subset \P \cap \tau^{-1}(T_j)$, an element $m \in \G^*$ and an injective $\delta$-map $\alpha: m D \to \P \cap \tau^{-1}(T_i)$ (where $\delta$ is a very small parameter to be determined later depending only on $\epsilon$ and $k$). 
We can then find $z\in\B\delta$ and $r\in\N$ such that $z^r=m$.
The overall idea is to bridge the gap between $D$ and $\P \cap \tau^{-1}(T_i)$ by considering a chain of maps
\begin{figure}[h]
    \centering
    \begin{tikzcd}[column sep=20pt]
        D \ar[r, rightarrow, "\beta_0"] \ar[d, hookrightarrow] & \lfloor z D \rfloor \ar[r, rightarrow, "\beta_1"] \ar[d, rightarrow, "\tilde{\alpha}_1"] & \lfloor z^2 D \rfloor \ar[r, rightarrow, "\beta_2"] \ar[d, rightarrow, "\tilde{\alpha}_2"] & \ldots \ar[r, rightarrow, "\beta_{r-1}"] & \lfloor z^r D \rfloor = m D \ar[r, rightarrow, "\alpha"] \ar[d, rightarrow, "\tilde{\alpha}_r"] & \P \cap \tau^{-1}(T_i) \\
        \P \cap \tau^{-1}(T_j) & \P \cap \tau^{-1}(T_{\ell_1}) & \P \cap \tau^{-1}(T_{\ell_2}) & \ldots & \P \cap \tau^{-1}(T_{\ell_r})
    \end{tikzcd}
\end{figure}

The parameter $\delta$ is chosen small enough so that $\tilde{\alpha}_{t+1} \circ \beta_t \circ \tilde{\alpha}_t^{-1}$ is an $\epsilon$-map from a divergent subset of $\P \cap \tau^{-1}(T_{\ell_t})$ to $\P \cap \tau^{-1}(T_{\ell_{t+1}})$. 
The same holds for $\tilde{\alpha}_1 \circ \beta_0$ and $\alpha \circ \tilde{\alpha}_r^{-1}$. Therefore, by inductively using property \eqref{quote}, we get $T_j \nu = T_{\ell_1} \nu = \ldots = T_{\ell_r} \nu = T_i \nu$.

We now formalize this construction.

 Since $z^r=m\in\G^*$ and $r\in\N$, we have that $|z|\geq 1$.
    Let $\chi=\chi_\delta$ be the coloring introduced in \cref{subsection_coloring}.

    In view of \cref{lemma_sparserties}, we can replace $D$ with a divergent subset which is $\delta$-sparse.
    \cref{lem:r_delta_sparse} then implies that $\lfloor z^t D\rfloor$ is $\delta$-sparse for all $t \in \{0, 1, \ldots, r\}$. 
We want that, for each $t \in  \{0,\dots, r\}$, the map $\beta_t\colon \lfloor z^t D \rfloor\to\lfloor z^{t+1} D\rfloor$, $\lfloor z^tn\rfloor\mapsto\lfloor z^{t+1}n\rfloor$ is a well defined, injective map. 
Since $\big|\lfloor z_1\rfloor-\lfloor z_2\rfloor\big|>|z_1-z_2|-2$ for any $z_1,z_2\in\C$, by replacing $D$ with a divergent subset satisfying $|d_1-d_2|>2r+2$ for every pair $d_1,d_2\in D$, we make sure that indeed each of the maps $\beta_t$ is injective and well defined.  
Since $z \in S_{\delta}$, after removing from $D$ a finite set if necessary, each of the maps $\beta_t$ is a $(2\delta)$-map.

    Using the fact that being divergent is a partition regular property, we can once again replace $D$ with a divergent subset so that the set $\lfloor z^t D\rfloor$ is $\chi$-monochromatic for each $t \in \{0, 1, \ldots, r\}$.
    Let $\ell_t\in[k]$ be the (single) color of $\lfloor z^t D\rfloor$ for each $t\in\{1,\dots r\}$.
    For convenience of notation, let $\ell_0:=j$, so that $\ell_0\in K$ trivially.
    We will prove by induction that $\ell_t\in K$ for all $t\in\{0,\dots,r\}$.

    By \cref{lemma_coloredprimes3}, after passing to a co-finite subset of $D$ if necessary, for each $t\in\{0,\dots,r\}$ there exists an injective $\delta$-map $\tilde\alpha_{t}:\lfloor z^t D\rfloor\to\P\cap\tau^{-1}(T_{\ell_t})$ (for $t=0$ we can take $\tilde\alpha_0$ to be the identity map).
    Making $\delta$ small enough in terms of $\epsilon$, it follows that $\tilde\alpha_{t+1}\circ\beta_t\circ\tilde\alpha_t^{-1}$ is an injective $\epsilon$-map between a divergent subset of $\P\cap\tau^{-1}(T_{\ell_t})$ and $\P\cap\tau^{-1}(T_{\ell_{t+1}})$ (where we only define $\tilde\alpha_t^{-1}$ on those numbers which has a (unique) pre-image).
    In view of \eqref{quote}, it follows by induction that $\ell_t\in K$ for all $t\in\{0,\dots,r\}$.
    
    In particular $\ell_r\in K$. 
    Since $z^rD=mD$ and $\alpha:mD\to\P\cap\tau^{-1}(T_i)$ is an injective $\delta$-map, it follows that $\alpha\circ \tilde\alpha_r^{-1}$ is an injective $\epsilon$-map between a divergent subset of $\P\cap\tau^{-1}(T_{\ell_r})$ and $\P\cap\tau^{-1}(T_i)$ (again we only define $\tilde\alpha_r^{-1}$ on those numbers which has a (unique) pre-image).
    Using \eqref{quote} one final time, we conclude that $i\in K$ as desired.

\end{proof}

\section{Applications} \label{Section_applications}
In this section we deduce from \cref{thm_finiterank} other results formulated in the introduction.
For a function $f\colon \G^{\ast} \to \C$, we say that $\E(f)$ exists if $\E_{n \in \Phi_N} f(n)$ converges to the same limit for every dilated F{\o}lner sequence $(\Phi_N)_{N \in \N}$. 
In this case, we use $\E(f)$ to denote the common limit.

\subsection{Comparing the averages of two multiplicative functions}

In this section we present some preliminary estimates needed for the proof of \cref{thm:generalized_Wirsing}.

\begin{lemma}\label{lem:bounded_by_f_g}
    For any dilated F{\o}lner sequence $(\Phi_N)_{N \in \N}$, there exist constants $C,N_0 > 0$ such that if $f, g: \G^{\ast} \to \C$ are completely multiplicative functions bounded by $1$, then for all $N\geq N_0$,
    \[
    \left| \E_{n \in \Phi_N} f(n) - g(n) \right| < C \sum_{p \in \P} \frac{|f(p) - g(p)|}{\NN(p)}.
    \]
\end{lemma}

\begin{proof}
Since $(\Phi_N)_{N \in \N}$ is a dilated F\o lner sequence we have $\Phi_N = \G^{\ast} \cap k_N U$ where $U$ is a Jordan measurable subset of $\C$ and $(k_N)_{N\in\N}$ is a sequence of positive real numbers satisfying $k_N \to \infty$ as $N \to \infty$. 
Let $R > 0$ be such that the disk $R\D$ contains $U$. 
Note that $\Phi_N \subset \G^*\cap k_N R\D$ for all $N$ and 
\[
    \lim_{N \to \infty} \frac{|\Phi_N|}{|\G^*\cap k_N R\D|} = \frac{\operatorname{m}(U)}{\operatorname{m}(R\D)} > 0,
\]
where $\operatorname{m}$ denotes the Lebesgue measure.
Thus there are positive constants $\rho,N_0$ such that $|\Phi_N| > \rho |\G^*\cap k_N R\D|$ for all $N\geq N_0$.
It follows that for any $a\in\G^{\ast}$,
\begin{equation}\label{eq_ballfolnerestimate}
\left|\Phi_N/a\right|
\leq
\big|(\G^{\ast}\cap k_N R\D)/a \big|
\leq
\frac{4k_N^2R^2}{\NN(a)}
\leq
\frac{4|\G^{\ast}\cap k_N R\D|}{\NN(a)}
\leq
\frac{4|\Phi_N|}{\rho\NN(a)}
\end{equation}
where the second inequality above follows from the fact that every element of $(\G^{\ast}\cap k_N R\D)/a$ is contained in the square $\{n \in \G^*: |\Re(n)|, |\Im(n)| \leq k_N R/|a|\}$, and the third inequality follows from the fact that every element of the square $\{n \in \G^*: |\Re(n)|, |\Im(n)| \leq k_N R/2\}$ is contained in $\G^{\ast}\cap k_N R\D$.

    Let $f,g:\G^\ast\to\C$ be completely multiplicative functions bounded by $1$. First, assume $f(q) = g(q)$ for all primes $q \in \P$ except for $q = p$. Let $C = 16/\rho$. We will show that for all $N\geq N_0$, 
    \begin{equation}\label{eq_oneprimedifference}
        \left| \E_{n \in \Phi_N} f(n) - g(n) \right| < \frac{C|f(p) - g(p)|}{\NN(p)}.
    \end{equation}
    Indeed, since $f(n) = g(n)$ if $p \nmid n$, we have
    \begin{multline*}
    \left|\sum_{n \in \Phi_N}\big(f(n)-g(n)\big)\right|
    =    \left|\sum_{k=0}^\infty\sum_{\substack{n \in \Phi_N \\ p^k \| n \\ }}\big(f(n)-g(n)\big)\right|
    =
    \left|\sum_{k=0}^\infty\sum_{\substack{n \in \Phi_N/p^k \\ p \nmid n}}f(np^k)-g(np^k)\right|
    \\=
    \left|\sum_{k=0}^\infty\big(f(p^k)-g(p^k)\big) \sum_{\substack{n \in \Phi_N/p^k \\ p \nmid n}}f(n)\right|
    \leq
    \sum_{k=1}^\infty\big|f(p^k)-g(p^k)\big|\cdot\big|\Phi_N/p^k\big|,
    \end{multline*}
    where $p^{k}\| n$ means that $k$ is the largest non-negative integer with $p^{k}|n$.
    Using \eqref{eq_ballfolnerestimate} and dividing by $|\Phi_N|$, we get that
    \begin{multline*}
    \left|\E_{n \in \Phi_N}\big(f(n)-g(n)\big)\right|
    <
    \sum_{k=1}^\infty\frac{8|f(p)^k-g(p)^k|}{\rho \NN(p)^{k}}
    =
    \frac{4|f(p)-g(p)|}{\rho \NN(p)}\sum_{k=1}^\infty\frac{|\sum_{j=0}^{k-1}f(p)^j g(p)^{k-1-j}|}{\NN(p)^{k-1}}
    \\
    \leq
     \frac{4|f(p)-g(p)|}{\rho \NN(p)}\sum_{k=1}^\infty\frac k{\NN(p)^{k-1}} 
     \leq 
     \frac{4|f(p)-g(p)|}{\rho \NN(p)}\sum_{k=1}^\infty\frac k{2^{k-1}}
     \leq \frac{16 |f(p)-g(p)|}{\rho \NN(p)},
     \end{multline*}
     establishing \eqref{eq_oneprimedifference}.
     
     For arbitrary functions $f$ and $g$, we can apply \eqref{eq_oneprimedifference} by changing $f$ at one prime at a time. 
     More precisely, fix $N > N_0$ and let $\{p_1, \ldots, p_s\}\subset\P$ be a list of pairwise co-primes containing (up to a unit) every prime divisor of an element of $\Phi_N$. 
     Define completely multiplicative functions $f_j: \G^* \to \C$ inductively as follows: $f_0 = f$ and for each $1 \leq j \leq s$, let
     \[
        f_j(p) = \begin{cases}
            f_{j-1}(p) \text{ for } p \in \P \setminus \{p_{j}\}\\
            g(p) \text{ for } p = p_{j}.
        \end{cases}
     \]
     Since $f_j(p) = f_{j-1}(p)$ for all primes $p$ except for $p = p_{j}$, using \eqref{eq_oneprimedifference} we deduce that
     \[
        \left| \E_{n \in \Phi_N} (f_{j}(n) - f_{j-1}(n) ) \right| <  \frac{C|f(p_{j}) - g(p_{j})|}{\NN(p_{j})}.
     \]
     By construction, $f_{k}(n) = g(n)$ for all $n \in \Phi_N$. Thus, by the triangle inequality,
     \[
     \left| \E_{n \in \Phi_N} (f(n) - g(n)) \right| \leq \sum_{j=1}^{s} \left| \E_{n \in \Phi_N} (f_{j-1}(n) - f_{j}(n) ) \right| < C\sum_{j=1}^s \frac{|f(p_j) - g(p_j)|}{\NN(p_j)} < C \sum_{p \in \P} \frac{|f(p) - g(p)|}{\NN(p)}.
     \]

\end{proof}

\begin{lemma}\label{lem:one_difference}
    Let $f, g: \G^{\ast} \to \C$ be completely multiplicative functions bounded by $1$. Suppose $\E(g)$ exists and $f(q) = g(q)$ for all but one prime $q = p$ and $f(\i)=g(\i)$. Then $\E(f)$ exists and equals 
    \[
        \frac{1 - g(p)/\NN(p)}{1 - f(p)/\NN(p)} \cdot \E(g). 
    \]
\end{lemma}
\begin{proof}
    Fix an arbitrary dilated F{\o}lner sequence $(\Phi_N)_{N \in \N}$.
         \begin{eqnarray*}
          \sum_{n \in \Phi_N} f(n) 
          &=& 
          \sum_{n \in \Phi_N} f(n) \sum_{k=0}^\infty1_{p^k||n}=
          \sum_{k=0}^{\infty} f(p)^k \left(\sum_{n \in \Phi_N/p^k} g(n)1_{p\nmid n}\right)
          \\&=&
          \sum_{k=0}^\infty f(p)^k\left(\sum_{n \in \Phi_N/p^k} g(n)-g(p)\sum_{n \in \Phi_N/p^{k+1}} g(n)\right)
          \\&=&
          \sum_{n \in \Phi_N} g(n)+\big(f(p)-g(p)\big)\sum_{k=1}^\infty f(p)^{k-1}\sum_{n\in\Phi_N/p^k}g(n).
        \end{eqnarray*}
        Since $|\Phi_N/p^k|/|\Phi_N|\to\frac1{\NN(p)^k}$ as $N\to\infty$, it follows that, after dividing by $|\Phi_N|$ and taking the limit as $N\to\infty$ above (and making use of the dominated convergence theorem),
        $$\lim_{N\to\infty}\E_{n\in\Phi_N}f(n)=\E(g)+\frac{f(p)-g(p)}{\NN(p)}\cdot\sum_{k=1}^\infty \frac{f(p)^{k-1}}{\NN(p)^{k-1}}\E(g)
        =
        \frac{\NN(p)-g(p)}{\NN(p)-f(p)}\E(g).$$
    In particular, $\lim_{N\to\infty}\E_{n\in\Phi_N}f(n)$ exists and its value is independent of the F\o lner sequence $(\Phi_N)_{N \in \N}$, so $\E(f)$ exists.
\end{proof}

\begin{lemma}\label{lem:when_f_g_close}
    Let $f, g: \G^{\ast} \to \C$ be completely multiplicative functions bounded by $1$ satisfying $f(\i) = g(\i)$.
    Suppose $\E(g)$ exists and
    \begin{equation}\label{eq:distance_g_f_bounded}
        \sum_{p \in \P_1} \frac{|g(p) - f(p)|}{\NN(p)} < \infty.
    \end{equation}
    Then $\E(f)$ exists and equals
    \[
        \E(g) \cdot\prod_{p \in \P_1} \frac{1 - g(p)/\NN(p)}{1 - f(p)/\NN(p)}. 
    \]
\end{lemma}

\begin{proof}
First note that 
\begin{equation}\label{eq:product_exists}
        \prod_{p \in \P_1} \frac{1 - g(p)/\NN(p)}{1 - f(p)/\NN(p)}
        =
        \prod_{p \in \P_1}\left(1+ \frac{f(p)-g(p)}{\NN(p) - f(p)}\right)
    \end{equation}
    and hence converges in view of \eqref{eq:distance_g_f_bounded} and the fact that $|f(p)|\leq1\leq\NN(p)/2$ for every $p\in\P_1$.

    Fix $\epsilon>0$ and find a finite set $F\subset\P_1$ such that
    $$\left|\prod_{p\in F} \frac{1 - g(p)/\NN(p)}{1 - f(p)/\NN(p)}- \prod_{p \in \P_1} \frac{1 - g(p)/\NN(p)}{1 - f(p)/\NN(p)}\right| < \epsilon\qquad\text{and}\qquad \sum_{p\in\P_1\setminus F}\frac{|g(p)-f(p)|}{\NN(p)}<\epsilon.$$
    Let $h:\G^\ast\to\C$ be the completely multiplicative function satisfying $h(p)=f(p)$ for $p\in F$, $h(p)=g(p)$ for $p\in\P_1\setminus F$, and $h(\i) = 1$.
    Applying \cref{lem:one_difference} inductively $|F|$ times, it follows that $\E(h)$ exists and equals
    $$\E(g)\cdot\prod_{p\in F}\frac{1 - g(p)/\NN(p)}{1 - f(p)/\NN(p)}.$$

    On the other hand, in view of \cref{lem:bounded_by_f_g}, for every dilated F{\o}lner sequence $(\Phi_N)_{N \in \N}$, there is a constant $C > 0$ such that $\big|\E_{n\in\Phi_N} (h(n)-f(n)) \big|\leq C\epsilon$ for a sufficiently large $N$.
    It follows that
    $$\limsup_{N\to\infty}\left|\E_{n\in\Phi_N}f(n)- \E(g) \cdot\prod_{p \in \P_1} \frac{1 - g(p)/\NN(p)}{1 - f(p)/\NN(p)}\right|<\big(C+|\E(g)|\big)\epsilon.$$
    The conclusion now follows by taking $\epsilon\to0$.
\end{proof}

\begin{corollary}\label{cor_oldlemma_2}
            Let $g:\G^{\ast}\to[0,1]$ be a completely multiplicative function such that
    $\sum_{p\in\P_1}\frac{|1-g(p)|}{\NN(p)}=\infty$.
    Then $\E(g)=0$.
\end{corollary}

\begin{proof} 
For each $M\in\N$, let $g_M:\G^{\ast}\to[0,1]$ be the completely multiplicative function satisfying $g_M(\i) = 1, g_M(p)=g(p)$ whenever $\NN(p)\leq M$ and $g_M(p)=1$ otherwise.
Using \cref{lem:when_f_g_close} to compare $g_M$ with the constant $1$ function we conclude that $\E(g_M)$ exists and satisfies
$$\E(g_M)=\prod_{\substack{p\in\P_1 \\\NN(p)\leq M}} \frac{1-1/\NN(p)}{1-g(p)/\NN(p)}\leq \prod_{\substack{p\in\P_1 \\ \NN(p)\leq M}} \left(1-\frac{1-g(p)}{\NN(p)}\right).$$
Since $\sum_{p\in\P_1}\frac{|1-g(p)|}{\NN(p)}=\infty$ it follows that $\lim_{M\to\infty}\E(g_M)=0$.
On the other hand, for all $M$ we have $g_M(n)\geq g(n) \geq 0$, so we conclude that $\E(g)$ exists and equals $0$.
\end{proof}

\subsection{Proof of \texorpdfstring{\cref{thm:generalized_Wirsing}}{Theorem B}}
\label{sec:wirsing_gaussian}

Let $f:\G^* \to \C$ be a bounded completely multiplicative function.
Recall from \eqref{define:P_f} the infinite product $P(f)$:
\[
    P(f) \coloneqq  \prod_{p \in \P_1}\frac{\NN(p) - 1}{\NN(p) - f(p)}.
\]
We first show that, if $\Arg(f(\P))$ is finite, then $P(f)$ converges. 

\begin{lemma}\label{lem:existence_Pf}
    Let $f: \G^{\ast} \to \C$ be a completely multiplicative function bounded by $1$ and suppose $\Arg(f(\P_1))$ is finite. Then $P(f)$
    converges and has norm at most 1; moreover, $P(f) = 0$ if and only if 
    \[
        \sum_{p \in \P_1} \frac{|1 - f(p)|}{\NN(p)} = \infty.
    \]
\end{lemma}

\begin{proof}
    We will use the easy to check fact that $\prod(1+a_i)$ converges for any bounded complex valued sequence $(a_i)$ for which $\sum|a_i|$ is finite.
    Since
    $$\frac{\NN(p)-1}{\NN(p)-f(p)}=1-\frac{1-f(p)}{\NN(p)-f(p)},\qquad\text{ and }\qquad\left|\frac{1-f(p)}{\NN(p)-f(p)}\right|\leq2\left|\frac{1-f(p)}{\NN(p)}\right|$$
    it follows that $P(f)$ converges whenever $\displaystyle \sum_{p \in \P_1} \frac{|1 - f(p)|}{\NN(p)} < \infty$.

    On the other hand, if $\displaystyle \sum_{p \in \P_1} \frac{|1 - f(p)|}{\NN(p)} = \infty$ and $\Arg(f(\P))$ is finite, there must exist $r<1$ such that $Q:=\{p\in\P_1:\Re f(p)<r\}$ is divergent. We will show that, in this case, $P(f)=0$ (and in particular it exists).    
    Since $\left|\frac{\NN(p)-1}{\NN(p)-f(p)}\right|\leq1$ for all $p$, if suffices to show that $\prod_{p\in Q}\frac{\NN(p)-1}{\NN(p)-f(p)}=0$.

    Note that 
    $$\frac{\NN(p)-1}{\NN(p)-f(p)}=1-\frac{1-f(p)}{\NN(p)-f(p)} \qquad\text{ and }\qquad\Re\frac{1-f(p)}{\NN(p)-f(p)}
    \geq
    \Re\frac{1- f(p)}{2\NN(p)}
    \geq
    \frac{1-r}{2\NN(p)}$$ for every $p\in Q$. Because $Q$ is divergent, the desired conclusion is now true thanks to the fact that $\prod(1-a_i)=0$ whenever $(a_i)$ is a sequence of complex number bounded by $1$ satisfying $\sum\Re a_i=\infty$.
\end{proof}
By appealing to \cref{thm_finiterank}, we first prove \cref{thm:generalized_Wirsing} in the case when $f$ takes values in the unit circle. 
\begin{lemma}\label{lemma_applyingTheoremA}
    Let $f:\G^{\ast}\to\C$ be a completely multiplicative function taking values in the unit circle such that $f(\i) = 1$ and $f(\P)$ is finite.
    Then $\E(f)$ exists and equals $P(f)$.
\end{lemma}

\begin{proof}
Since $f(\P)$ is finite, $f(\P_1)$ is also finite. 
If the set $\{p\in\P_1:f(p)\neq1\}$ is convergent, then the conclusion follows immediately from applying \cref{lem:when_f_g_close} with $g \equiv 1$.
Otherwise, because $f(\P_1)$ is finite, it follows from \cref{lem:existence_Pf} that $P(f)=0$. Therefore, our goal is to show that, when the set $\{p\in\P_1:f(p)\neq1\}$ is divergent, $\E(f)$ exists and equals $0$.
We do this by invoking \cref{thm_finiterank}.

Let $X_0$ be the (finite) set of those points $z\in S^1$ in the unit circle for which the set $\{p\in\P_1:f(p)=z\}$ is divergent and let $X$ be the closed subgroup of $S^1$ generated by $X_0$.
Using the assumption that the set $\{p\in\P_1:f(p)\neq1\}$ is divergent, we deduce that $X$ is not the singleton $\{1\}$.
We ``remove'' the exceptional primes by considering the completely multiplicative function $g:\G^*\to\C$ defined by $g(p)=f(p)$ for any prime $p\in\P_1$ with $f(p)\in X_0$, $g(p)=1$ for any other prime $p$, and $g(\i) = 1$.
Since $f(\P_1)$ is finite, it follows that $f$ and $g$ only differ on a convergent set of primes and hence, using again \cref{lem:when_f_g_close}, it suffices to show that $\E(g)$ exists and equals $0$.

Since $X$ is the closed group generated by $g(\P_1)$, the Haar measure $\mu$ on $X$ is the unique measure on $X$ preserved under all the maps $\tau(n):x\mapsto g(n)x$ for $n\in\G^*$. 
Letting $F:X\to\C$ be the identity function and $x=1$, it follows from \cref{thm_finiterank} that $\E(g)$ exists and equals $\displaystyle\int_Xz\d\mu(z)$.
Since $\mu$ is the Haar measure on a subgroup of $S^1$ other than the trivial group $\{1\}$, the integral is $0$, finishing the proof.
\end{proof}

We are now ready to prove \cref{thm:generalized_Wirsing}.
\begin{proof}[Proof of \cref{thm:generalized_Wirsing}]
Let $f:\G^{\ast}\to\C$ be a bounded completely multiplicative such that the set $\Arg(f(\P))$ is finite.

If $\displaystyle \sum_{p \in \P_1} \frac{1 - |f(p)|}{\NN(p)} = \infty$, then by \cref{cor_oldlemma_2} we have $\E(|f|)=0$ which implies that $\E(f)=0$.
Since $|1 - f(p)| \geq 1 - |f(p)|$, we also have $\displaystyle \sum_{p \in \P_1} \frac{|1 - f(p)|}{\NN(p)} = \infty$ and hence by \cref{lem:existence_Pf}, $P(f) = 0$ as well.

If $\displaystyle \sum_{p \in \P_1} \frac{1 - |f(p)|}{\NN(p)} < \infty$, then we can approximate $f$ by the completely multiplicative function taking values in the unit circle $g(n):=e^{\i \Arg(f(n))}$.
Indeed, 
\begin{equation}\label{eq_fclosetog}
    \sum_{p \in \P_1} \frac{|g(p) - f(p)|}{\NN(p)} = \sum_{p \in \P_1} \frac{|1 - f(p)/g(p)|}{\NN(p)} = \sum_{p\in\P_1}\frac{1 - |f(p)|}{\NN(p)} < \infty.
\end{equation}
Moreover, since $\Arg(f(\P_1))$ is finite, the range $g(\P_1)$ is finite, and so \cref{lemma_applyingTheoremA} implies that $\E(g)$ exists and equals $P(g)$.
By \cref{lem:when_f_g_close}, $\E(f)$ also exists and equals $P(f)$.

\end{proof}

From \cref{thm:generalized_Wirsing}, one can easily deduce an analogous result for arbitrary bounded completely multiplicative function $f$ without the restriction that $f(\i) = 1$. 
Let $Q_1 = \{z \in \G: \Re(z) > 0, \Im(z) \geq 0\}$ denote the first quadrant of the complex plane $\C$. Let $Q_2 = \i Q_1, Q_3 = \i^2 Q_1$ and $Q_4 = \i^3 Q_1$ be the other quadrants. Let $\operatorname{m}$ be the Lebesgue measure on $\C$ and note that since $f$ is completely multiplicative, $f(\i)$ must be either $1, -1, \i$ or $-\i$.

\begin{theorem}\label{thm:no_restrictionC}
    Let $f: \G^*\to \C$ be a bounded completely multiplicative function such that $\Arg(f(\P))$ is finite.
    If $(\Phi_N)_{N \in \N}$ is a dilated F{\o}lner sequence corresponding to the Jordan measurable set $U$, then the limit $\lim_{N \to \infty} \E_{n \in \Phi_N} f(n)$ exists and equals 
    \[
         P(f) \cdot \sum_{k=0}^3  \frac{\operatorname{m}(U \cap Q_{k+1})}{\operatorname{m}(U)} \cdot f(\i)^k .
    \]
\end{theorem}

\begin{proof}
If $(\Phi_N)_{N \in \N}$ is a dilated F{\o}lner sequence in the first quadrant $Q_1$, the value of $f(\i)$ does not affect the average $\E_{n \in \Phi_N} f(n)$. Therefore, in this case, \cref{thm:generalized_Wirsing} implies that
\[
    \lim_{N \to \infty} \E_{n \in \Phi_N} f(n) = P(f).
\]

Now, let $(\Phi_N)_{N \in \N}$ be an arbitrary dilated F{\o}lner sequence in $\G$ and $U$ be a Jordan measurable set associated to $(\Phi_N)_{N \in \N}$. For $k \in \{0, 1, 2, 3\}$, we have
\[
    \E_{n \in \Phi_N \cap Q_{k+1}} f(n) = \E_{n \in \Phi_N \cap Q_{k+1}} f(\i^k) f(- \i^k n) = f(\i)^k \E_{n \in -\i^k (\Phi_N \cap Q_{k+1})} f(n) 
\]
which converges to $f(\i)^k P(f)$ as $N \to \infty$.
Here we use the fact that $(-\i^k (\Phi_N \cap Q_{k+1}))_{N \in \N}$ is a dilated F{\o}lner sequence in the first quadrant. 
Lastly, observe that
\[
    \lim_{N \to \infty} \E_{n \in \Phi_N} f(n) = \sum_{k = 0}^3 \frac{\operatorname{m}(U \cap Q_{k+1})}{\operatorname{m}(U)} \cdot \lim_{N \to \infty} \E_{n \in \Phi_N \cap Q_{k+1}} f(n)
\]
and so our theorem follows. 
\end{proof}

\subsection{Proofs of Theorems \ref{thm:convergence_integer}, \ref{thm:convergence_Gaussian1} and \ref{thm:convergence_Gaussian}}

The implications \cref{thm:generalized_Wirsing} $\Rightarrow$ \cref{thm:convergence_Gaussian} $\Rightarrow$ \cref{thm:convergence_Gaussian1} are easy and have been outlined in the introduction. To prove \cref{thm:convergence_integer}, given a bounded completely multiplicative function $f: \N \to \R$, we apply \cref{thm:convergence_Gaussian1} to the function $f \circ \NN: \G^* \to \R$. As mentioned in \cref{sec:notation}, a Gaussian integer $a + b\i$ is in $\P_1$ if and only if either:
\begin{itemize}
    \item $b=0$ and $a = p$ is a prime in $\N$ of the form $4n + 3$, or 
    \item $b>0$ and $a^2 + b^2 = p$ is a prime number in $\N$ (which will not be of the form $4n + 3)$. 
\end{itemize}
In the first case, the norm of the Gaussian prime $a + b \i$ is $p^2$ and in the second case, the norm is $p$. For each integer prime $p$ of the form $4n + 1$, there are two Gaussian primes in $\P_1$ associated to it: $a + b \i$ and $b + a \i$. Now it is simple to  check that $P(f \circ \NN)$ equals the expression in \eqref{eq:Pf_expanded_A}.

\subsection{Proof of \texorpdfstring{\cref{thm:main-Omega-m2n2}}{Theorem D}}

\label{sec:proof_omega_m2n2}
Suppose $(X, T)$ is a uniquely ergodic system with the unique invariant measure $\mu$. 
Denote by ${\mathcal T}$ the semigroup of all continuous self-maps of $X$. 
For $n \in \G^{\ast}$, define 
\[
    \tau(n) = T^{\Omega(\NN(n))}.
\]
Then $\tau: (\G^*, \times) \to{\mathcal T}$ is a semigroup homomorphism because 
\[
    \tau(mn) = T^{\Omega(\NN(mn))} = T^{\Omega(\NN(m) \NN(n))} = T^{\Omega(\NN(m)) + \Omega(\NN(n))} = \tau(m)\circ \tau(n).
\]
Gaussian primes $p\in\P$ fall into three categories: 
\begin{itemize}
    \item $\NN(p)$ is a prime in $\Z$ which is $\equiv1\bmod4$;
    \item $\NN(p)$ is the square of a prime in $\Z$ which is $\equiv3\bmod4$;
    \item $\NN(p)=2$.
\end{itemize}
The last two categories forms a convergent set of Gaussian primes (since those $p$ in the second category are either real or purely imaginary), so for all $p\in\P$ outside a convergent set we have that $\NN(p)$ is an integer prime, and hence $\tau(p)=T$.
Therefore, by \cref{thm_finiterank}, for any $x \in X$ and $F \in C(X)$, 
\[
    \lim_{N \to \infty} \frac{1}{N^2} \sum_{m,n = 1}^N F(T^{\Omega(m^2 + n^2)} x) = \lim_{N \to \infty} \frac{1}{N^2} \sum_{z \in \Phi_N} F(\tau(z) x) = \int_X F \d \mu,
\]
where $(\Phi_N)_{N\in\N}$ is the dilated F\o lner sequence given by $\Phi_N:=\{z\in\G^{\ast}\colon 0<\Re z,\Im z\leq N\}$.

\section{Open questions}

\label{sec:open_questions}

As mentioned in the introduction, one of the main motivations for this paper was the question of Frantzikinakis and Host, \cref{question_FH}, which roughly speaking asks whether the Ces\`aro average of a real valued bounded completely multiplicative function over a homogeneous polynomial of two variables must always exist. 
Our \cref{thm:convergence_integer} gives a positive answer for the specific polynomial $P(m,n)=m^2+n^2$, and while our approach might be adaptable to handle other norm forms, the case of general polynomials remains out of reach.

In a similar vein, using \cref{thm:main-Omega-m2n2} as motivation, we ask whether the polynomial $P(m,n)=m^2+n^2$ can be replaced with more general polynomials:
\begin{question}
Let $(X, T)$ be a uniquely ergodic system. Let $P\in\Z[x,y]$ be a homogeneous polynomial taking values on the positive integers and let $F\in C(X)$.
Does the limit
\[
    \lim_{N\to\infty}\E_{1 \leq m, n \leq N}F(T^{\Omega(P(m,n))}x)
\]
exist?
\end{question}

In \cite{Loyd}, Loyd showed that an analogue of Birkhoff pointwise ergodic theorem along $\Omega(n)$ is false. 
More precisely, in every non-atomic ergodic system $(X, \mu, T)$, there is a measurable set $A \subset X$ such that for $\mu$-almost every $x \in X$, 
\[
    \limsup_{N \to \infty} \E_{1 \leq n \leq N} 1_A(T^{\Omega(n)} x) = 1 \text{ and } \liminf_{N \to \infty} \E_{1 \leq n \leq N} 1_A(T^{\Omega(n)} x) = 0.
\]
In light of \cref{thm:main-Omega-m2n2}, we can ask a similar question regarding the averages along $\Omega(m^2 + n^2)$:
\begin{question}
    Let $(X, \mu, T)$ be a non-atomic ergodic system. Is it true that there exists a measurable set $A \subset X$ such that for $\mu$-almost every $x \in X$,
    \[
         \limsup_{N \to \infty} \E_{1 \leq m, n \leq N} 1_A(T^{\Omega(m^2 + n^2)} x) = 1 \text{ and } \liminf_{N \to \infty} \E_{1 \leq m, n \leq N} 1_A(T^{\Omega(m^2+ n^2)} x) = 0?
    \]
\end{question}

In the same paper, Loyd \cite{Loyd} proved that Bergelson and Richter's result (\cref{thm:bergelson_richter}) would be false in general if we removed the unique ergodicity assumption. 
It was shown that there exists an ergodic system $(X, \mu, T)$, a $\mu$-generic point $x \in X$ and a continuous function $F \in C(X)$ such that
\[
    \lim_{N \to \infty} \E_{n \in [N]} F(T^{\Omega(n)} x)
\]
does not exist. This result suggests the following question:
\begin{question}
    Does there exist an ergodic system $(X, \mu, T)$, a $\mu$-generic point $x \in X$ and a continuous function $F \in C(X)$ such that
\[
    \E_{1 \leq m, n \leq N} F(T^{\Omega(m^2 + n^2)} x)
\]
does not converge as $N \to \infty$?
\end{question}

Theorems \ref{thm:convergence_Gaussian} and \ref{thm:generalized_Wirsing} are true for dilated F{\o}lner sequences. 
However, as discussed in \cref{sec:def_dilated_Folner_seq} and also \cref{sec:counter_example_non_dilated}, the conclusions of these theorems do not hold for every additive F{\o}lner sequence $(\Phi_N)$ in $\G$. 
This raises the following open-ended question:

\begin{question}\label{ques:which_folner_seq}
    Which additive F{\o}lner sequences $(\Phi_N)_{N \in \N}$ in $\G^{\ast}$ are such that for every bounded real-valued completely multiplicative function $f: \G^{\ast}\to \R$, the averages $\E_{n \in \Phi_N} f(n)$
    converge?
\end{question}

The sequence $\Phi_N = N + \{n\in\G: |\Re n|,|\Im n| <  N\}$ is a dilated F{\o}lner sequence corresponding to the open set $U = \{z \in \G: |\Re z-1|, |\Im z| < 1\}$. 
On the other hand, if we shift each square by $N^{1+\epsilon}$ we destroy this property.
Thus, a concrete instance of \cref{ques:which_folner_seq} is following:

\begin{question}\label{question_shifteddilatedFolner}
    Is there $\epsilon>0$ such that, letting $\Phi_N = N^{1+\epsilon} + \{n\in\G: |\Re n|,|\Im n| <  N\}$, for every bounded completely multiplicative function $f: \G^* \to \R$ the average $\E_{n \in \Phi_N} f(n)$
    converges?
\end{question}

In a similar direction, and in light of \cref{thm:convergence_Gaussian} and a recent  result on averages of multiplicative functions on short intervals of Matom\"{a}ki and Radziwi\l\l~\cite{Matomaki_Radziwill16}, we ask the following question: 

\begin{question}
    Let $f:\G^*\to\R$ be a bounded completely multiplicative function and let $\Phi_N=\{n:|\Re n|, |\Im n| \leq N\}$. Define $P(f)$ as in \eqref{define:P_f}.
    Is it true that
$$\lim_{H\to\infty}\lim_{N\to\infty}\E_{n\in\Phi_N}\left|\E_{h\in n+\Phi_H}f(h)-P(f)\right|=0?$$
More generally, does this hold for any dilated F\o lner sequence $(\Phi_N)_{N \in \N}$?
\end{question}

In the last three questions, it also makes sense to ask whether the restriction that the completely multiplicative functions involved are real valued can be relaxed to allow completely multiplicative functions $f:\G^*\to\C$ satisfying $\Arg(f(\P))<\infty$.

\appendix

\section{Some estimates}\label{sec_appendixA}
\begin{lemma}[Tur\'an-Kubilius]
\label{2.5}
Let $(\Phi_N)_{N\in\N}$ be a F\o lner sequence in $(\G,+)$ and let $B\subset \G^{\ast}$ be finite and non-empty. 
For any function $a\colon\G\to\mathbb{C}$ bounded by 1, we have that
$$\limsup_{N\to\infty}\left\vert\E_{n\in \Phi_{N}}a(n) - \E_{p\in B}^{\log}\E_{n\in \Phi_N/p}a(pn)\right\vert
\leq 
\left(\E_{p\in B}^{\log}\E_{q\in B}^{\log}\oldphi pq\right)^{1/2}.$$ 
\end{lemma}
\begin{proof}
Denote by $d(A):=\lim_{N\to\infty}\E_{n\in\Phi_N}1_A(n)$ when the limit exists.
Note that, for every $p,q\in\G^{\ast}$, 
$$d(p\G)=\frac1{\NN(p)}\qquad\text{ and }\qquad d(p\G\cap q\G)=\frac{\NN(\gcd(p,q))}{\NN(p)\NN(q)}.$$
For any finite set $B\subset\G^{\ast}$ and any $n\in\G$,
\begin{eqnarray*}
\left|\E^{\log}_{p\in B}1-\NN(p)\cdot1_{p\G}(n)\right|^2
&=&
\E^{\log}_{p\in B}\E^{\log}_{q\in B}(1-\NN(p)\cdot1_{p\G}(n))(1-\NN(q)\cdot1_{q\G}(n))
\\&=&
\E^{\log}_{p\in B}\E^{\log}_{q\in B}1+\NN(p)\NN(q)1_{p\G\cap q\G}(n)-\NN(p)1_{p\G}(n)-\NN(q)1_{q\G}(n)
\end{eqnarray*}
so averaging over $n$ we deduce that
$$\lim_{N\to\infty}\E_{n\in\Phi_N}\left|\E^{\log}_{p\in B}1-\NN(p)\cdot1_{p\G}(n)\right|^2
=
\E^{\log}_{p\in B}\E^{\log}_{q\in B}\NN(\gcd(p,q))-1$$
An application of the Cauchy-Schwarz inequality then yields the estimate
\begin{equation}\label{eq_proof_TturanKubilius}   \limsup_{N\to\infty}\left\vert\E_{n\in \Phi_{N}}a(n) - \E_{p\in B}^{\log}\NN(p)\E_{n\in \Phi_N}a(n)1_{p\G}(n)\right\vert
\leq 
\left(\E_{p\in B}^{\log}\E_{q\in B}^{\log}\oldphi pq\right)^{1/2}. 
\end{equation}
Finally notice that
$$\E_{n\in\Phi_N}a(n)1_{p\G}(n)
=
\frac1{|\Phi_N|}\sum_{n\in\Phi_N\cap p\G}a(n)
=
\frac{|\Phi_N\cap p\G|}{|\Phi_N|}\E_{n\in\Phi_N/p}a(np),$$
and since $d(p\G)=1/\NN(p)$, \eqref{eq_proof_TturanKubilius} is equivalent to the desired conclusion.
\end{proof}

To achieve the last condition in \cref{lemma_cleanedTK} we use the following estimates.

\begin{lemma}\label{lemma_gcdestimate}
    Let $F_1,F_2\subset\P$ be finite sets of primes. Then
    \begin{enumerate}[label=(\roman*)]
    \item $\displaystyle\E_{n,m\in F_1}^{\log}\NN(\gcd(n,m)) < 1 + \frac{4}{\log(F_1)}$.
    \item $\displaystyle\E_{n,m\in F_1F_2}^{\log}\NN(\gcd(n,m))<\left(1+\frac{8}{\log(F_1)}\right)\left(1+\frac{8}{\log(F_2)}\right)$.
    \end{enumerate}
\end{lemma}
\begin{proof}\
    \begin{enumerate}[label=(\roman*)]
        \item Denote by $U=\{1,i,-1,-i\}$ the group of units on $\G$.
        Note that for $m, n \in \P$, 
        \[
        \NN(\gcd(m,n)) = \begin{cases} \NN(m) \text{ if } m \in nU \\
        1 \text{ otherwise.} \end{cases}
        \]
        Therefore, for every $n\in F_1$,
        $$\E^{\log}_{m\in F_1}\NN(\gcd(n,m))
        =
        \frac1{\log(F_1)}\sum_{m\in F_1}\frac{\NN(\gcd(n,m))}{\NN(m)}
        <
        \frac1{\log(F_1)}\left(4+\sum_{m\in F_1}\frac{1}{\NN(m)}\right)
        =
        1 + \frac{4}{\log(F_1)}
        $$
        Averaging in $n\in F_1$ yields the desired result.
        
        \item Given $n,m\in F_1F_2$ we can decompose $n=n_1n_2$ and $m=m_1m_2$ with $n_i,m_i\in F_i$.
        Note that 
 $$\NN(\gcd(n,m))=\begin{cases}
            \NN(m)&\text{if }m_1 \in n_iU \text{ and } m_2\in n_{3-i}U \\
            \NN(m_j)&\text{if }m_j\in n_iU \text{ and }m_{3-j}\notin n_{3-i}U
            \\
            1&\text{otherwise}
        \end{cases}$$
        
        For a fixed $n\in F_1F_2$, partition $F_1F_2$ into the following $3$ sets: 
        $$A_1=\{m\in F_1F_2:\NN(\gcd(n,m))=\NN(m)\},\qquad A_2:=\{m\in F_1F_2:\NN(\gcd(n,m))=\NN(m_j)\},$$
        $$A_3:=\{m\in F_1F_2:\NN(\gcd(n,m))=1\}.$$ 
        We can then estimate the sum
        \begin{eqnarray*}
            \sum_{m\in F_1F_2}\frac{\NN(\gcd(m, n))}{\NN(m)}
            &=&
            \left(\sum_{m\in A_1}+\sum_{m\in A_2}+\sum_{m\in A_3}\right)\frac{\NN(\gcd(m, n))}{\NN(m)}
            \\&\leq&            
            32+8\big(\log(F_1)+\log(F_2)\big)+\log(F_1F_2).
        \end{eqnarray*}
        Dividing by $\log(F_1F_2)=\log(F_1)\log(F_2)$ we conclude that for each $n\in F_1F_2$ we have
        $$\E_{m\in F_1F_2}^{\log}\NN(\gcd(m, n))< \left(1+\frac8{\log(F_1)}\right)\left(1+\frac8{\log(F_2)}\right).$$
        Averaging over $n\in F_1F_2$ yields the desired result. 
    \end{enumerate}
\end{proof}

The following estimate is needed for the proof of \cref{lem:xyzuu}:
\begin{lemma}\label{lemma_estimateaverages}
    Let $F$ be a non-empty finite set and let $v,w:F\to\R^{>0}$ and $f,g:F\to\C$.
    Suppose that $|f(p)|\leq1$ for all $p\in F$ and that there exists $\delta>0$ such that for each $p\in F$,
    $$\left|\frac{w(p)}{v(p)}-1\right|<\delta,\qquad \text{ and }\qquad\big|f(p)-g(p)\big|<\delta.$$
    Then
    $$\left|\frac{\sum_{p\in F}w(p)f(p)}{\sum_{p\in F}w(p)} - \frac{\sum_{p\in F}v(p)g(p)}{\sum_{p\in F}v(p)}\right|<3\delta.$$
\end{lemma}
\begin{proof}
    All sums in the proof are over $p\in F$, so we will omit this subscript.
    First note that
    \begin{equation}\label{eq_lemma_estimateaverages1}
        \left|\frac{\sum v(p)g(p)}{\sum v(p)}-\frac{\sum v(p)f(p)}{\sum v(p)}\right|
    \leq
    \frac1{\sum v(p)}\sum v(p)\big|g(p)-f(p)\big|<\delta.
    \end{equation}
    Second we have $\big|w(p)-v(p)\big|\leq\delta v(p)$ and hence
    \begin{equation}\label{eq_lemma_estimateaverages2}\left|\frac{\sum v(p)f(p)}{\sum v(p)}-\frac{\sum w(p)f(p)}{\sum v(p)}\right|
    =
    \frac{\left|\sum \big(v(p)-w(p)\big)f(p)\right|}{\sum v(p)}
    \leq
    \frac{\sum \delta v(p)\big|f(p)\big|}{\sum v(p)}
    \leq\delta.
    \end{equation}
    Finally we have
    \begin{eqnarray*}
    \left|\frac{\sum w(p)f(p)}{\sum w(p)} - \frac{\sum w(p)f(p)}{\sum v(p)}\right|
    &=&
    \left|\sum w(p)f(p)\right|\cdot\left|\frac1{\sum w(p)} - \frac1{\sum v(p)}\right|
    \\&\leq&
    \left(\sum w(p)\right)\cdot\left|\frac1{\sum w(p)} - \frac1{\sum v(p)}\right|
    \\&=&
    \frac{\left|\sum v(p)-w(p)\right|}{\sum v(p)}
    \leq
    \frac{\delta \sum v(p)}{\sum v(p)}=\delta.
    \end{eqnarray*}
    Together with \eqref{eq_lemma_estimateaverages1} and \eqref{eq_lemma_estimateaverages2} this yields the desired conclusion.
\end{proof}

\section{A counterexample with non-dilated F{\o}lner sequences}

\label{sec:counter_example_non_dilated}

A function $f: \G^{\ast} \to \{-1, 1\}$ is called \emph{normal} if every finite pattern of $-1$ and $1$ appears in $f$ at the correct frequency; that is for all $k \in \N$, all distinct $h_1, \ldots, h_k \in \G$, and $\epsilon_1, \ldots, \epsilon_k \in \{-1, 1\}$, the set
\[
    S := \big\{n \in \G: f(n + h_j) = \epsilon_j \text{ for all } j = 1, \ldots, k\big\}
\]
satisfies
\[
    \lim_{N \to \infty} \frac{|S \cap N\D|}{|\G\cap N\D|} = \frac{1}{2^k}, 
\]
where we write $f(0)=0$ for convenience.
In this section, to demonstrate the necessity of dilated F{\o}lner sequences in Theorems \ref{thm:convergence_Gaussian}, \ref{thm:generalized_Wirsing}, and \ref{thm_finiterank}, we show that a ``random'' completely multiplicative function $f: \G^{\ast} \to \{-1, 1\}$ is almost surely normal. 
Note that the averages of a normal function $f$ fail to converge for some F\o lner sequence.   
Indeed, normality implies, for instance, that there exist arbitrarily large disks in which $f$ equals 1, and similarly, arbitrarily large disks where $f$ equals $-1$.
By considering a F\o lner sequence that consists of such disks and alternates between the value $1$ and $-1$, it follows that the averages of $f$ do not converge.

The analogous result over $\N$, that a random completely multiplicative function $\N\to\{-1,1\}$ is normal, was proved by Fish \cite{Fish-normal_Liouville} and the proof extends to Gaussian integers without any major difficulty; we provide a proof here for completeness.

\begin{lemma}[{\cite[Lemma 2.2]{Fish-normal_Liouville}}] \label{lem:dont_need_N}
    Let $(a_n)_{n \in \N}$ be a bounded sequence. Let $T_N = \frac{1}{N} \sum_{n = 1}^N a_n$ and $t \in \C$. The followings are equivalent:
    \begin{enumerate}
        \item $\lim_{N \to \infty} T_N = t$,

        \item There exists a sequence of increasing indices $(N_j)_{j \in \N}$ such that $N_{j+1}/N_j \to 1$ and $\lim_{j \to \infty} T_{N_j} = t$.
    \end{enumerate}
\end{lemma}

\begin{proposition}\label{prop:random_Liouville_Chowla}
     Define a random completely multiplicative function $f: \G^{\ast} \to \{-1, 1\}$ as follows: let $f(i)=1$, for every $p \in \P$, let $f(p) = 1$ or $-1$ with probability $1/2$ each and independent from each other. For convenience write $f(0)=0$. Then almost surely,
     \[
        \lim_{N \to \infty} \E_{\substack{n\in\G \\ \NN(n) < N}} f(n + h_1) \cdots f(n + h_k) = 0
     \]
     for all $k \in \N$ and pairwise distinct $h_1, \ldots, h_k \in \G$. 
\end{proposition}

\begin{proof}
In this proof, we use $\E$ for averages and $E$ for expected values of random variables.

Fix $k$ and distinct $h_1, \ldots, h_k \in \G$. For $n \in \G$, define
\[
    \xi(n) := (n + h_1) \cdots (n + h_k)
\]
and
\[
    \phi(n) := f(\xi(n)),
\]
and for $N \in \N$, define
\[
    T_N = \E_{\substack{n\in\G \\ \NN(n) < N}} \phi(n).
\]
Our goal is to show $T_N \to 0$ almost surely. In order to do this, we will show
\[
    \sum_{N = 1}^{\infty} E(T_{N^{40}}^2) < \infty.
\]
Note that
\[
    T_N^2 = \E_{\substack{x,y\in\G \\ \NN(x) < N \\ \NN(y) < N}} \phi(x) \phi(y).
\]
Therefore, by the linearity of expectation, we get
\[
    E(T_N^2) = \E_{\substack{x,y\in\G \\ \NN(x) < N \\ \NN(y) < N}} E (\phi(x) \phi(y)).
\]
If $x \in \G^{\ast}$ is not a square, write $x=p_1,\ldots,p_km^2$ for distinct primes $p_1,\ldots,p_k$, and note that $f(x)=f(p_1,\ldots,p_k)$. Using the law of total probabilities (conditioning, for instance, in the values of $f(p_1)$) and the independence of the values of $f(p_i)$, we can deduce that the probability of $f(x)=1$ and $f(x)=-1$ are both equal to 1/2. It follows that
\[
    E(\phi(x) \phi(y)) = \begin{cases}
        1 \text{ if } \xi(x) \xi(y) \text{ is a square} \\
        0 \text{ otherwise}.
    \end{cases}
\]

Thus, in order to bound $E(T_N^2)$, we need to bound the number of pairs $(x, y) \in B_N(0)$ such that $\xi(x) \xi(y)$ is a square. 

For any $x \in \G^{\ast}\cap N\D$, write $\xi(x) =q_1 q_2 \cdots q_{\ell} m^2$ where   $q_1, q_2, \ldots, q_{\ell}$ are distinct primes. Define $h(x) = \ell$, the number of prime divisors of the square-free component of $\xi(x)$. Let $D$ be the set of all possible Gaussian integers that divide at least two numbers in $\{n + h_1, \ldots, n + h_k\}$. Thus $D$ is the set of all possible divisors of $h_{j_1} - h_{j_2}$ for $1 \leq j_1 < j_2 \leq k$. For a finite subset $S \subset \G$, denote by $m(S)$ the product of all elements in $S$ and define $m(\varnothing) = 1$. 

With the above $x$, if $\xi(x) \xi(y)$ is a square, there exists $S_1 \subset D$ and $S_2 \subset \{q_1, \ldots, q_{\ell}\}$ such that $y = m(S_1) m(S_2) n^2$ for some $n \in \G$. Thus, the number of $y \in B_N(0)$ such that $\xi(x) \xi(y)$ is a square is at most
\[
    2^{|D|} \cdot 2^{h(x)} |B_{\sqrt{N}}(0)| \ll 2^{h(x)} \sqrt{N}.
\]
Thus
\[
    E(T_N^2) \ll \frac{1}{N^{3/2}} \sum_{x \in \G\cap N\D} 2^{h(x)}.
\]
For any positive integer $M$, if all prime divisors of $\xi(x)$ have norms greater than $M$, then 

\[
    h(x) \leq \log_M (\NN(x + h_1)\cdots \NN(x+h_k) < k \max_{1\leq i \leq k}\log_M \NN(x + h_i).
\]
Fix an $M$ such that $M > 2^{k/0.45}$ and let $C$ be the number of primes whose norms do not exceed $M$. Then for any $x \in \G$,

\[
    h(x) < C + k \max_{1\leq i \leq k}\log_M \NN(x + h_i) .
\]
It follows that
\[
    2^{h(x)} < 2^C \cdot \max_{1\leq i \leq k}  2^{k\log_M \NN(x + h_i)} < 2^C \cdot\max_{1\leq i \leq k} \NN(x + h_i)^{0.45}\ll N^{0.45}.
\]
Therefore, 

\[
    E(T_N^2) \ll \frac{1}{N^{3/2}} \sum_{x \in N\D} 2^{h(x)} \ll \frac{1}{N^{0.05}},
\] 
and so
\[
    \sum_{N=1}^{\infty} E(T_{N^{40}}^2) \ll \sum_{n = 1}^{\infty} \frac{1}{N^2} < \infty.
\]
Thus $T_{N^{40}} \to 0$ almost surely and by \cref{lem:dont_need_N}, $T_N \to 0$ almost surely.
\end{proof}

The next lemma is about the equivalence of Chowla's conjecture and the normality of the Liouville function. This result is well-known in $\Z$ and its proof in $\G$ is the same. However, since we could not locate the exact proof in the literature, we include it for completeness.
\begin{lemma}\label{lem:chowla_normal_Liouville}
    If $f: \G \to \{-1, 1\}$ satisfies
    \begin{equation}\label{eq:f_chowla_hypothesis}
        \lim_{N \to \infty} \E_{\substack{n\in\G \\ \NN(n) < N}} f(n+h_1) \cdots f(n + h_k) = 0
    \end{equation}
    for every $k \in \N$ and distinct $h_1, \ldots, h_k \in \G$, then $f$ is normal.
\end{lemma}
\begin{proof}
    Let $X = \{-1, 1\}^{\G}$ and $T$ be the $\G$-action on $X$ defined by $T_g(x(n)) = x(n + g)$ for all $g \in \G$ and $x \in X$. Let $\mu$ be a weak$^*$-limit of the sequence of measures
    \[
        \E_{\substack{g\in\G \\ \NN(g) < N}} \delta_{T_g f} 
    \]
    where $\delta_{T_g f}$ is the Dirac measure at $T_g f$.
    Define a function $F: X \to \{-1, 1\}$ by $F(x) = x(0)$ for all $x \in \G$. Then by \eqref{eq:f_chowla_hypothesis} and the definition of $\mu$, for all $k \in \N$ and distinct $h_1, \ldots, h_k \in \G$,
    \[
        \int_X T^{h_1} F \cdots T^{h_k} F \d \mu = 0 = \left( \int_X F \d \mu \right)^k.
    \]
    Since $C(X)$ contains a dense subset which is generated by the family of functions of the form $T^{h_1} F \cdots T^{h_k} F$, we deduce that $\mu$ is the Bernoulli measure on $X$. Thus for any $k \in \N$, distinct $h_1, \ldots, h_k \in \G$, and $\epsilon_1, \ldots, \epsilon_k \in \{-1, 1\}$, the cylinder set 
    \[
        C := \{x \in X: x(h_j) = \epsilon_j \text{ for all } j \in [k] \}
    \]
    has measure $\mu(C) = 1/2^k$. By the definition of $\mu$, it means that that pattern of $-1, 1$ appears in $f$ at the correct frequency.
\end{proof}

\cref{prop:random_Liouville_Chowla} and \cref{lem:chowla_normal_Liouville} imply the proposition:

\begin{proposition}\label{prop:random_Liouville}
     Define a random completely multiplicative function $f: \G^{\ast} \to \{-1, 1\}$ as follows: let $f(\i)=1$ and, for every prime $p \in \P_1$, let $f(p) = 1$ or $-1$ with probability $1/2$ each  and independent from each other. Then almost surely, $f$ is normal.
\end{proposition}

\bibliographystyle{plain}

\begin{thebibliography}{10}
	
	\bibitem{Bergelson_Richter_2020}
	V.~Bergelson and F.~Richter.
	\newblock Dynamical generalizations of the prime number theorem and
	disjointness of additive and multiplicative semigroup actions.
	\newblock {\em Duke Math. J.}, 171(15):3133--3200, 2022.
	
	\bibitem{Chowla-Riemann-hypothesis}
	S.~Chowla.
	\newblock {\em The {R}iemann hypothesis and {H}ilbert's tenth problem}.
	\newblock Mathematics and its Applications, Vol. 4. Gordon and Breach Science
	Publishers, New York-London-Paris, 1965.
	
	\bibitem{VallePoussin:1896}
	Ch.~J. de~la Vall\'ee-Poussin.
	\newblock Recherches analytiques sur la th\'eorie des nombres premiers.
	\newblock {\em Brux.S .sc.20 B}, pages 363--397, 1896.
	
	\bibitem{VallePoussin:1897}
	Ch.~J. de~la Vall\'ee-Poussin.
	\newblock Recherches analytiques sur la th\'eorie des nombres premiers.
	\newblock {\em Brux.S .sc.20 B}, pages 251--342, 1897.
	
	\bibitem{Delange-arithmetic}
	H.~Delange.
	\newblock On some arithmetical functions.
	\newblock {\em Illinois J. Math.}, 2:81--87, 1958.
	
	\bibitem{Delange-surles}
	H.~Delange.
	\newblock Sur les fonctions arithm\'{e}tiques multiplicatives.
	\newblock {\em Ann. Sci. \'{E}cole Norm. Sup. (3)}, 78:273--304, 1961.
	
	\bibitem{Donoso-Le-Moreira-Sun-additive-averages}
	S.~Donoso, A.~Le, J.~Moreira, and W.~Sun.
	\newblock Additive averages of multiplicative correlation sequences and
	applications.
	\newblock {\em J. Anal. Math.}, 149(2):719--761, 2023.
	
	\bibitem{Erdos-Some-unsolved-problems-1957-Michigan-Math}
	P.~Erd\H{o}s.
	\newblock Some unsolved problems.
	\newblock {\em Michigan Math. J.}, 4:291--300, 1957.
	
	\bibitem{Ferenczi_Kulaga-Przymus_Lemanczyk18}
	S.~Ferenczi, J.~Ku{\l}aga-Przymus, and M.~Lema\'{n}czyk.
	\newblock Sarnak's conjecture: what's new.
	\newblock In {\em Ergodic theory and dynamical systems in their interactions
		with arithmetics and combinatorics}, volume 2213 of {\em Lecture Notes in
		Math.}, pages 163--235. Springer, Cham, 2018.
	
	\bibitem{Fish-normal_Liouville}
	A.~Fish.
	\newblock Random {L}iouville functions and normal sets.
	\newblock {\em Acta Arith.}, 120(2):191--196, 2005.
	
	\bibitem{Frantzikinakis-Host_Asymp}
	N.~Frantzikinakis and B.~Host.
	\newblock Asymptotics for multilinear averages of multiplicative functions.
	\newblock {\em Math. Proc. Cambridge Philos. Soc.}, 161(1):87--101, 2016.
	
	\bibitem{Frantzikinakis_Host_2017}
	N.~Frantzikinakis and B.~Host.
	\newblock Higher order {F}ourier analysis of multiplicative functions and
	applications.
	\newblock {\em J. Amer. Math. Soc.}, 30(1):67--157, 2017.
	
	\bibitem{Furstenberg81}
	H.~Furstenberg.
	\newblock {\em Recurrence in ergodic theory and combinatorial number theory}.
	\newblock Princeton University Press, Princeton, N.J., 1981.
	
	\bibitem{Granville-Harper-Soundararajan-2015}
	A.~Granville, A.~Harper, and K.~Soundararajan.
	\newblock Mean values of multiplicative functions over function fields.
	\newblock {\em Res. Number Theory}, 1:Paper No. 25, 18, 2015.
	
	\bibitem{Green_Tao10}
	B.~Green and T.~Tao.
	\newblock Linear equations in primes.
	\newblock {\em Ann. of Math. (2)}, 171(3):1753--1850, 2010.
	
	\bibitem{Halasz-1968}
	G.~Hal\'{a}sz.
	\newblock \"{U}ber die {M}ittelwerte multiplikativer zahlentheoretischer
	{F}unktionen.
	\newblock {\em Acta Math. Acad. Sci. Hungar.}, 19:365--403, 1968.
	
	\bibitem{Hecke1918}
	E.~Hecke.
	\newblock Eine neue {A}rt von {Z}etafunktionen und ihre {B}eziehungen zur
	{V}erteilung der {P}rimzahlen.
	\newblock {\em Math. Z.}, 1(4):357--376, 1918.
	
	\bibitem{Helfgott-thesis}
	H.~Helfgott.
	\newblock {\em Root numbers and the parity problem}.
	\newblock ProQuest LLC, Ann Arbor, MI, 2003.
	\newblock Thesis (Ph.D.)--Princeton University.
	
	\bibitem{Helfgott-Parity-reducible}
	H.~Helfgott.
	\newblock The parity problem for reducible cubic forms.
	\newblock {\em J. London Math. Soc. (2)}, 73(2):415--435, 2006.
	
	\bibitem{Klurman-Mangerel-effective}
	O.~Klurman and A.~Mangerel.
	\newblock Effective asymptotic formulae for multilinear averages of
	multiplicative functions.
	\newblock arXiV:1708.03176.
	
	\bibitem{Landau1903}
	E.~Landau.
	\newblock Neuer {B}eweis des {P}rimzahlsatzes und {B}eweis des
	{P}rimidealsatzes.
	\newblock {\em Math. Ann.}, 56(4):645--670, 1903.
	
	\bibitem{Loyd}
	K.~Loyd.
	\newblock A dynamical approach to the asymptotic behavior of the sequence.
	\newblock {\em Ergodic Theory and Dynamical Systems}, pages 1--22, 2022.
	
	\bibitem{Matomaki_Radziwill16}
	K.~Matom\"{a}ki and M.~Radziwi\l\l.
	\newblock Multiplicative functions in short intervals.
	\newblock {\em Ann. of Math. (2)}, 183(3):1015--1056, 2016.
	
	\bibitem{Matthiesen-linearcorrelations}
	L.~Matthiesen.
	\newblock Linear correlations of multiplicative functions.
	\newblock {\em Proc. Lond. Math. Soc. (3)}, 121(2):372--425, 2020.
	
	\bibitem{Pillai-generalization-mangoldt}
	S.~Pillai.
	\newblock Generalisation of a theorem of {M}angoldt.
	\newblock {\em Proc. Indian Acad. Sci., Sect. A.}, 11:13--20, 1940.
	
	\bibitem{Rudnick-Waxman-AnglesOfPrimes}
	Z.~Rudnick and E.~Waxman.
	\newblock Angles of {G}aussian primes.
	\newblock {\em Israel J. Math.}, 232(1):159--199, 2019.
	
	\bibitem{Selberg-zur}
	S~Selberg.
	\newblock Zur {T}heorie der quadratfreien {Z}ahlen.
	\newblock {\em Math. Z.}, 44(1):306--318, 1939.
	
	\bibitem{Sun1}
	W.~Sun.
	\newblock A structure theorem for multiplicative functions over the gaussian
	integers and applications.
	\newblock {\em J. Analyse Math.}, 134(1):55--105, 2018.
	
	\bibitem{Sun2}
	W.~Sun.
	\newblock Sarnak’s conjecture for nilsequences on arbitrary number fields and
	applications.
	\newblock {\em Adv. in Math.}, 415:p.108883, 2023.
	
	\bibitem{Wirsing_61}
	E.~Wirsing.
	\newblock Das asymptotische {V}erhalten von {S}ummen \"{u}ber multiplikative
	{F}unktionen.
	\newblock {\em Math. Ann.}, 143:75--102, 1961.
	
	\bibitem{Wirsing-1967}
	E.~Wirsing.
	\newblock Das asymptotische {V}erhalten von {S}ummen \"{u}ber multiplikative
	{F}unktionen. {II}.
	\newblock {\em Acta Math. Acad. Sci. Hungar.}, 18:411--467, 1967.
	
\end{thebibliography}

\end{document}